\RequirePackage{luatex85}
\documentclass{amsart}
\usepackage{etex}
\usepackage{fixltx2e}

\usepackage[english,french]{babel}

\usepackage[usenames,dvipsnames]{xcolor}
\usepackage[bookmarksopen,bookmarksdepth=2]{hyperref}
\hypersetup{colorlinks=true,citecolor=NavyBlue,linkcolor=BrickRed,urlcolor=Green} 
\usepackage[mathscr]{eucal}
\usepackage{tikz}
\usepackage{tikz-cd}
\usepackage{tikzsymbols}
\usepackage{adjustbox}
\usepackage{enumitem}
\usepackage{dsfont}

\usepackage{microtype}

\renewcommand{\mathbb}{\mathbf}
\usepackage{amssymb,amsmath,amsfonts,amsthm,epsfig,amscd}
\usepackage{stmaryrd}
\usepackage[all,cmtip,poly]{xy}
\usepackage{color}

\newcommand{\triv}{\mathrm{triv}}

\newcommand{\red}{\operatorname{red}}

\newcommand{\WBT}{W^{\operatorname{BT}}}
\newcommand{\BT}{\operatorname{BT}}
\newcommand{\ocZ}{\overline{\cZ}}
\newcommand{\cZbar}{\overline{\cZ}}
\newcommand{\cTbar}{\overline{\cT}}
\newcommand{\Rbarinfty}{{\bar{R}_\infty}}

\newcommand{\JH}{\operatorname{JH}}
 \newcommand{\sigmabar   }{\overline{\sigma}}   
\def\iso{\buildrel \sim \over \longrightarrow}

\newtheorem{thm}[subsection]{Theorem}
\newtheorem{lemma}[subsection]{Lemma}
\newtheorem{lem}[subsection]{Lemma}

\newtheorem{cor}[subsection]{Corollary}

\newtheorem{prop}[subsection]{Proposition}

\newtheorem{alemma}[subsection]{Lemma}

\newtheorem{aprop}[subsection]{Proposition}

\theoremstyle{definition}

\theoremstyle{remark}
\newtheorem{remark}[subsection]{Remark}
\newtheorem{rem}[subsection]{Remark}

\def\numequation{\addtocounter{subsubsection}{1}\begin{equation}}
\def\nummultline{\addtocounter{subsubsection}{1}\begin{multline}}
\def\anumequation{\addtocounter{subsection}{1}\begin{equation}}
\def\anummultline{\addtocounter{subsection}{1}\begin{multline}}
\renewcommand{\theequation}{\arabic{section}.\arabic{subsection}.\arabic{subsubsection}}

\newif\iffinalrun
\finalruntrue
\iffinalrun
\else
\fi

\iffinalrun
  \newcommand{\need}[1]{}
  \newcommand{\mar}[1]{}
\else
  \newcommand{\need}[1]{{\tiny *** #1}}
  \newcommand{\mar}[1]{\marginpar{\raggedright\tiny fixme #1}}
\fi

\newcommand{\F}{\FF}

\newcommand{\Q}{\QQ}

\newcommand{\Z}{\ZZ}

\newcommand{\p}{\frakp}

\newcommand{\FF}{{\mathbb F}}

\newcommand{\QQ}{{\mathbb Q}}

\newcommand{\ZZ}{{\mathbb Z}}

\renewcommand{\bf}{\ensuremath{\mathbf{f}}}

\newcommand{\cC}{{\mathcal C}}

\newcommand{\cF}{{\mathcal F}}

\newcommand{\cO}{{\mathcal O}}
\newcommand{\cP}{{\mathcal P}}

\newcommand{\cR}{{\mathcal R}}

\newcommand{\cT}{{\mathcal T}}
\newcommand{\cU}{{\mathcal U}}
\newcommand{\cV}{{\mathcal V}}
\newcommand{\cW}{{\mathcal W}}
\newcommand{\cX}{{\mathcal X}}
\newcommand{\cY}{{\mathcal Y}}
\newcommand{\cZ}{{\mathcal Z}}

\newcommand{\frakp}{\mathfrak{p}}

\newcommand{\tB}{\mathrm{B}}

\newcommand{\Fbar}{\overline{\F}}
\newcommand{\Qbar}{\overline{\Q}}
\newcommand{\Zbar}{\overline{\Z}}

\newcommand{\Fp}{\F_p}

\newcommand{\Fpbar}{\Fbar_p}

\newcommand{\Fpbarx}{\Fpbar^{\times}}

\newcommand{\Zpbar}{\Zbar_p}

\newcommand{\Qp}{\Q_p}

\newcommand{\Qpbar}{\Qbar_p}

\DeclareMathOperator{\Ext}{Ext}

\DeclareMathOperator{\Gal}{Gal}
\DeclareMathOperator{\GL}{GL}

\DeclareMathOperator{\Spec}{Spec}

\DeclareMathOperator{\Spf}{Spf}

\DeclareMathOperator{\Sym}{Sym}

\newcommand{\cris}{\mathrm{cris}}

\newcommand{\St}{\mathrm{St}}

\newcommand{\Bcris}{\tB_{\cris}}

\newcommand{\Dpcris}{\operatorname{D_{pcris}}}

\newcommand{\rhobar}{\overline{\rho}}

\newcommand{\dirlim}{\varinjlim}
 
\newcommand{\into}{\hookrightarrow}

\newcommand{\toisom}{\buildrel\sim\over\to}

\newcommand{\dd}{\mathrm{dd}}

\newcommand{\rbar}{\overline{r}}

\newcommand{\st}{\mathrm{ss}}

\newcommand{\cegsBtresramlemma}{\cite[Lem.~A.5]{cegsB}}

\newcommand{\cegsBdisjointunion}{\cite[Cor.~4.2.13]{cegsB}}

\newcommand{\cegsBlastsection}{\cite[Sec.~5.2]{cegsB}}
\newcommand{\cegsBCmainresultsplus}{\cite[Cor.~3.1.8, Cor.~4.5.3, Prop.~5.2.21]{cegsB}}
\newcommand{\cegsBZmainresultsplus}{\cite[Thm.~5.1.2, Prop.~5.1.4, Lem.~5.1.8, Prop.~5.2.20]{cegsB}}
\newcommand{\cegsBversalprop}{\cite[Prop.~5.2.19]{cegsB}}
\newcommand{\cegsCrecollections}{\cite[Sec.~2.3]{cegsC}}
\newcommand{\cegsCnotation}{\cite[Sec.~1.2]{cegsC}}
\newcommand{\cegsCcomponentdescription}{\cite[Prop.~5.1.13, Thm.~5.1.17, Cor.~5.3.3, Thm.~5.4.3]{cegsC}}
\newcommand{\cegsCbigfamilies}{\cite[Thm.~5.1.17]{cegsC}}
\newcommand{\cegsCverticalcomponents}{\cite[Thm.~5.1.12]{cegsC}}
\newcommand{\cegsCextensionssec}{\cite[Sec.~4.2.7]{cegsC}}
\newcommand{\cegsBfieldchange}{\cite[Prop.~4.3.1(2)]{cegsB}}
\newcommand{\cegsCmodeldefn}{\cite[Def.~5.1.16]{cegsC}}

\begin{document}
\selectlanguage{english}
\title [The geometric  Breuil--M\'ezard conjecture]{The geometric Breuil--M\'ezard conjecture for  two-dimensional
potentially Barsotti--Tate Galois representations}

\author[A. Caraiani]{Ana Caraiani}\email{a.caraiani@imperial.ac.uk}
\address{Department of Mathematics, Imperial College London, London SW7 2AZ, UK}

\author[M. Emerton]{Matthew Emerton}\email{emerton@math.uchicago.edu}
\address{Department of Mathematics, University of Chicago,
5734 S.\ University Ave., Chicago, IL 60637, USA}

\author[T. Gee]{Toby Gee} \email{toby.gee@imperial.ac.uk} \address{Department of
  Mathematics, Imperial College London,
  London SW7 2AZ, UK}

\author[D. Savitt]{David Savitt} \email{savitt@math.jhu.edu}
\address{Department of Mathematics, Johns Hopkins University}
\thanks{The first author was supported in part by NSF grant DMS-1501064, 
  by a Royal Society University Research Fellowship,
  and by ERC Starting Grant 804176. The second author was supported in part by the
  NSF grants DMS-1303450, DMS-1601871,
  DMS-1902307, DMS-1952705, and DMS-2201242. 
 The third author was 
  supported in part by a Leverhulme Prize, EPSRC grant EP/L025485/1, Marie Curie Career
  Integration Grant 303605, 
  ERC Starting Grant 306326, and a Royal Society Wolfson Research
  Merit Award. The fourth author was supported in part by NSF CAREER
  grant DMS-1564367 and  NSF grant 
  DMS-1702161.}

\begin{abstract}
 We establish a geometrisation of the Breuil--M\'ezard conjecture for 
 potentially Barsotti--Tate representations, 
 as well as of the weight part of Serre's conjecture, for 
 moduli stacks of two-dimensional mod~$p$ representations of the
 absolute Galois group of a $p$-adic local field. These results are first proved
for the stacks of \cite{cegsB,cegsC}, and then transferred to the stacks of \cite{EGmoduli}
 by means of a comparison of versal rings.
\end{abstract}
\maketitle

\selectlanguage{english}

\setcounter{tocdepth}{1}
\tableofcontents

\section{Introduction} 

 Let~$K/\Qp$ be a finite extension with residue field
$k$,
let~$\overline{K}$ be an algebraic closure of~$K$, and let $d \ge 1$
be a positive integer. Two of us (M.E.\ and T.G.\ \cite{EGmoduli}) have
constructed moduli stacks of representations of the absolute Galois group
$G_K := \Gal(\overline{K}/K)$, globalizing Mazur's classical deformation
theory of Galois representations. These stacks are expected to be the backbone of
 a categorical $p$-adic Langlands correspondence, playing
the role anticipated by the stacks of \cite{DHKM,ZhuStacks} in the $\ell \neq p$ setting.

To be precise, the book \cite{EGmoduli} defines the
category $\mathcal{X}_d$ fibred in groupoids over $\Spf \Z_p$ whose
$A$-valued points, for any $p$-adically complete $\Z_p$-algebra $A$,
are the groupoid of rank $d$ projective \'etale
$(\varphi,\Gamma)$-modules with $A$-coefficients. Then the finite type
points of $\mathcal{X}_d$ correspond to representations $\rbar : G_K
\to \GL_d(\Fpbar)$, the versal rings of $\mathcal{X}_d$ at finite type points recover classical
Galois deformation rings, and one has the following, which is one of
the main results of \cite{EGmoduli}.

\begin{thm}{\cite[Thm.~1.2.1]{EGmoduli}}\label{thm:intro-thm-EG}
Each $\mathcal{X}_d$ is a Noetherian formal algebraic stack. Its underlying
reduced substack $\mathcal{X}_{d,\red}$ is an algebraic stack of
finite type over~$\Fp$, and is equidimensional of dimension $[K:\Qp]
\binom{d}{2}$. The irreducible components of $\mathcal{X}_{d,\red}$
have a natural bijective labeling by  Serre weights.   
\end{thm}

Recall that a \emph{Serre weight} in this context is an irreducible
$\Fpbar$-representation of $\GL_d(k)$ (or rather an
isomorphism class thereof). The description in \cite{EGmoduli}  of
the labeling of  components of
$\mathcal{X}_{d,\red}$ by Serre weights  is to some extent combinatorial. Namely, it is shown that each irreducible component of
$\mathcal{X}_{d,\red}$ has a dense set of $\Fpbar$-points
which are successive extensions of characters, with extensions as
non-split as possible.  The restrictions of these characters to the
inertia group yield discrete data (their tame inertia weights) which,
together with some further information about peu and tr\`es ramifiee
extensions, amounts precisely to the data of (a highest weight of)
a Serre weight.

It is expected, however, that there is another  description of the irreducible components
of $\mathcal{X}_{d,\red}$ that is  more precise and more informative from the perspective of the
$p$-adic Langlands program. The weight part of Serre's conjecture, as
described for instance in \cite[Sec.\ 3]{MR3871496},
associates to each $\rbar : G_K \to \GL_d(\Fpbar)$ a set of Serre
weights $W(\rbar)$. One expects for each Serre weight $\sigma$ that
there is a set of components of $\mathcal{X}_{d,\red}$, including the
irreducible component labelled by~$\sigma$ in~\cite{EGmoduli}, 
the union of whose $\Fpbar$-points are \emph{precisely} the representations $\rbar$ with $\sigma \in W(\rbar)$.
Equivalently, after adding additional
labels to some of the components (so that components will be labelled
by a set of weights, rather than a single weight), the set
$W(\rbar)$ is precisely the collection of labels of the various
components of $\mathcal{X}_{d,\red}$ on which $\rbar$ lies. 

One of the aims of this paper is to establish this expectation
in the case $d=2$, taking as input the weight part of Serre's conjecture for $\GL_2$ \cite{gls13} and the Breuil--M\'ezard conjecture for two-dimensional potentially Barsotti--Tate representations \cite{geekisin}, and thus obtain a description of
\emph{all} the finite type points of each irreducible component of 
$\mathcal{X}_{2,\red}$, as opposed to just a dense set
of points. Indeed we have the following theorem, which can be regarded
as a geometrisation of the weight part of Serre's conjecture for
$\GL_2$. If $\sigma$ is a Serre weight, let
$\mathcal{X}^{\sigma}_{d,\red}$ denote the irreducible component of
$\mathcal{X}_{d,\red}$ labelled by $\sigma$ in~\cite{EGmoduli}. (We refer the reader to Section~\ref{subsec: notation} for any unfamiliar notation or terminology in what follows.)

\begin{thm}\label{thm:component-thm} Suppose $p > 2$. For each Serre weight $\sigma$ we define a cycle $Z^{\sigma}$ as follows:
\begin{itemize}
  \item $Z^{\sigma} = \mathcal{X}_{2,\red}^\sigma$, if the weight $\sigma$ is not Steinberg, while
  \item $Z^{\chi\otimes \St} = \mathcal{X}_{2,\red}^\chi + \mathcal{X}_{2,\red}^{\chi \otimes \St}$ if the weight  $\sigma \cong \chi \otimes \St$ is Steinberg.
\end{itemize}
Then $\sigma \in W(\rbar)$ if and only if $\rbar$ lies in the support of $Z^{\sigma}$.

Indeed a stronger statement is true:\ the cycles $Z^{\sigma} = \mathcal{X}_{2,\red}^\sigma$ {\upshape(}for $\sigma$ non-Steinberg{\upshape)} and $Z^{\chi\otimes \St} = \mathcal{X}_{2,\red}^\chi + \mathcal{X}_{2,\red}^{\chi \otimes \St}$ constitute the cycles in a geometric version of the Breuil--M\'ezard conjecture {\upshape(}to be explained below{\upshape)}. 
\end{thm}

We emphasise that the existence of such a geometric interpretation of
the sets~$W(\rbar)$ is far from
obvious, and indeed we know of no direct proof using any of the
explicit descriptions of~$W(\rbar)$ in the literature; it seems hard to understand in any explicit way which Galois
representations arise as the limits of a family of extensions of given
characters, and the description of the sets~$W(\rbar)$ is
very complicated (for example, the description
in~\cite{bdj} relies on certain $\Ext$ groups of crystalline
characters). Our proof is indirect, and ultimately makes use of a
description of~$W(\rbar)$ given in~\cite{geekisin}, which is in terms
of potentially Barsotti--Tate deformation rings of~$\rbar$ and is
motivated by the Taylor--Wiles method. We interpret this description
in the geometric language of~\cite{emertongeerefinedBM}, which we 
in turn interpret as the formal completion of a ``geometric
Breuil--M\'ezard conjecture'' for our stacks.

The proof of Theorem~\ref{thm:component-thm} entwines the main results of the book \cite{EGmoduli} with the results of our papers \cite{cegsB,cegsC}. Indeed Theorem~\ref{thm:component-thm} is (more or less) stated at \cite[Thm.~8.6.2]{EGmoduli}, but the argument given there makes reference to (an earlier version of) this paper\footnote{The reference [CEGS19, Thm.~5.2.2] in \cite{EGmoduli} is Theorem~\ref{thm:stack version of geometric Breuil--Mezard} of this paper, while the reference [CEGS19, Lem.~B.5] in \cite{EGmoduli} is \cegsBtresramlemma.}.  We should therefore explain more precisely what are the contributions of this paper.

For each Hodge type $\lambda$ and inertial type $\tau$, the book \cite{EGmoduli} constructs a closed substack $\mathcal{X}_d^{\lambda,\tau}\subset \mathcal{X}_d$ parameterizing $d$-dimensional potentially crystalline representations of $G_K$ of Hodge type $\lambda$ and inertial type $\tau$. When $d=2$, $\lambda$ is trivial, and $\tau$ is tame, these are stacks of potentially Barsotti--Tate representations of type $\tau$, and we write $\mathcal{X}_2^{\tau,\BT}$ instead. 

The papers \cite{cegsB,cegsC} construct and study another stack $\cZ^{\dd}$ which can  be regarded as a stack of \emph{tamely} potentially Barsotti--Tate representations;\ as well as a closed substack $\cZ^{\tau} \subset \cZ^{\dd}$, for each tame type $\tau$, of potentially Barsotti--Tate representations of type $\tau$. 
Our stacks $\cZ^{\tau}$ are presumably\footnote{\emph{Added in revision}: In fact this has now been proved:\ see \cite[Thm~4.5]{APAW}.} isomorphic to the stacks
  $\mathcal{X}_2^{\tau,\BT}$, but literally they are different stacks, constructed differently:\ the stacks $\cZ^{\tau}$ are stacks of \'etale $\varphi$-modules with tame descent data, constructed by taking the scheme-theoretic image of a stack $\cC^{\tau,\BT}$ of Breuil--Kisin modules with tame descent data;\ whereas the $\mathcal{X}_2^{\tau,\BT}$ are stacks of \'etale $(\varphi,\Gamma)$-modules, constructed by taking the scheme-theoretic image of a stack of Breuil--Kisin--Fargues modules satisfying a descent condition.  In practice it seems to be easier to compute with the stacks $\cZ^{\tau}$ than the stacks $\mathcal{X}_2^{\tau,\BT}$.

The properties of $\cZ^{\dd}$ and $\cZ^{\tau}$ that we will use in this paper are recalled in detail in Section~\ref{sec:recollections}, but we mention two crucial properties now.
\begin{itemize}
  \item  It is proved in \cite{cegsB} by a local model argument that the special fibre of $\cC^{\tau,\BT}$ is reduced. As a consequence so is its scheme-theoretic image $\cZ^{\tau,1}$ in $\cZ^{\tau}$. The stack $\cZ^{\dd,1}$, the scheme-theoretic image in $\cZ^{\dd}$ of the special fibre of $\cC^{\dd,\BT}$, is similarly reduced.
  \item  It is shown in \cite{cegsC} that the irreducible components of $\cZ^{\tau,1}$ are in bijection with the Jordan--H\"older factors of $\sigmabar(\tau)$;\ the component corresponding to $\sigma$ has a dense set of $\Fpbar$-points $\rbar$ such that $W(\rbar) = \{\sigma\}$.
\end{itemize}
Here $\sigma(\tau)$ is the  representation of $\GL_2(\cO_K)$ corresponding to $\tau$ under the inertial local Langlands correspondence, and $\sigmabar(\tau)$ is its reduction modulo $p$. Note the similarity between the second of these two properties, and the labeling by Serre weights in Theorem~\ref{thm:intro-thm-EG}. 

These properties are combined in Section~\ref{subsec: generically
  reduced} to prove that the special fibre of $\cZ^{\tau}$ is generically reduced. (Note that in general the special fibre of $\cZ^{\tau}$ need not be the same as $\cZ^{\tau,1}$;\ and similarly for the special fibre of $\cZ^{\dd}$ vis-\`a-vis $\cZ^{\dd,1}$.) From this we deduce the following theorem about the special fibres of potentially Barsotti--Tate deformation rings, which seems hard to prove purely in the setting of
formal deformations. Let $\cO$ be the ring of integers in a finite extension of $\Qp$, with residue field $\F$.

  \begin{thm}\label{thm:intro=gen-red} Let $\rbar : G_K \to \GL_2(\F)$ be a continuous representation, and $\tau$ a tame type. Let $R^{\tau,\BT}_{\rbar}$ be the universal framed deformation $\cO$-algebra parameterising potentially Barsotti--Tate lifts of $\rbar$ of type $\tau$. Then $R^{\tau,\BT}_{\rbar}\otimes_{\cO} \F$ is generically reduced.  
  \end{thm}

We anticipate that this result will be of independent interest.  For example, one of us (A.C.), in joint work \cite{CaraianiNewton} with James Newton, has used this result in the proof of a modularity lifting theorem in the Barsotti--Tate case for $\GL_2$ over a CM field;\ this modularity lifting theorem is used, in turn, to deduce the modularity of elliptic curves over $\Q(\sqrt{-d})$ for $d \in \{1,2,3,5\}$.

We remark that tr\`es ramifi\'ee
representations do not have tamely potentially Barsotti--Tate lifts,
hence do not correspond to finite type points on
$\cZ^{\dd,1}$. Equivalently (\emph{cf}.~\cegsBtresramlemma), the Jordan--H\"older factors of
$\sigmabar(\tau)$ for tame types $\tau$ are never Steinberg, and
therefore the stacks $\cZ^{\tau,1}$ and $\cZ^{\dd,1}$ do not have
irreducible components corresponding to Steinberg weights. So,
although $\cZ^{\dd,1}$ and $\cX_{2,\red}$ cannot be isomorphic, we
anticipate (but do not prove) that there is an isomorphism between
$\cZ^{\dd,1}$ and the union of the non-Steinberg components of
$\cX_{2,\red}$, along the same lines as \cite[Thm.~4.5]{APAW}.

If $\sigma$ is a non-Steinberg weight, then $\sigma$ can be written as a virtual linear combination of representations $\sigmabar(\tau)$ in the Grothendick group of $\GL_2(k)$. In Section~\ref{sec:appendix on geom BM} this observation is translated into a special case of the classical geometric Breuil--M\'ezard conjecture \cite{emertongeerefinedBM};\ we globalize this in Section~\ref{sec: picture} to prove the following theorem, which is the main result of this paper.

\begin{thm}\label{thm:intro-main-Z} The irreducible components of $\cZ^{\dd,1}$ are in bijection with non-Steinberg Serre weights;\ write  $\ocZ(\sigma)$ for the component corresponding to $\sigma$. Then:
\begin{enumerate}
  \item The finite type points of $\ocZ(\sigma)$ are precisely the representations $\rbar : G_K \to \GL_2(\F')$ having $\sigma$ as a Serre weight.
  \item The stack  $\cZ^{\tau,1}$ is equal to $\cup_{\sigma \in \JH(\sigmabar(\tau))} \ocZ(\sigma)$.
\end{enumerate}
\end{thm}
Part (1) of the theorem is the analogue of Theorem~\ref{thm:component-thm} for the stacks $\cZ^{\dd,1}$, while part (2) is a geometrisation of the Breuil--M\'ezard conjecture for our tamely potentially {}Barsotti--Tate stacks. Theorem~\ref{thm:intro=gen-red} is used crucially in the proof of Theorem~\ref{thm:intro-main-Z}, to confirm that each component $\ocZ(\sigma)$ contributes with multiplicity at most one to the cycle of $R^{\tau,\BT}\otimes_{\cO} \F$.  We emphasize that, to this point, our results are independent from those of \cite{EGmoduli}.

As explained above, our construction excludes the tr\`es ramifi\'ee
representations, which are twists of certain extensions of the trivial
character by the mod~$p$ cyclotomic character. From the point of view
of the weight part of Serre's conjecture, they are precisely the
representations which admit a twist of the Steinberg representation as
their only Serre weight. In accordance with the picture described
above, this means that the full moduli stack of 2-dimensional
representations of~$G_K$ can be obtained from our
stack by adding in the irreducible components consisting of the tr\`es
ramifi\'ee representations. This is carried out by extending our results to the stacks of~\cite{EGmoduli}, using the full strength of \emph{loc.\ cit.}

In particular, it is proved in~\cite{EGmoduli} that the classical (numerical) Breuil--M\'ezard conjecture is equivalent to a geometrised Breuil--M\'ezard conjecture for the stacks $\cX_d^{\lambda,\tau}$ of~\cite{EGmoduli}. Taking the Breuil--M\'ezard conjecture for potentially Barsotti--Tate representations \cite{geekisin} as input, they obtain the following theorem.

\begin{thm}[\cite{EGmoduli}]\label{thm:EG-BM} There exist effective cycles $Z^{\sigma}$ {\upshape{(}}elements of the free group on the irreducible components of $\cX_{2,\red}$, with nonnegative coefficients{\upshape{)}} such that for all inertial types $\tau$, the cycle of the special fibre of $\cX_{2}^{\tau,\BT}$ is equal to $\sum_{\sigma} m_{\sigma}(\tau) \cdot Z^{\sigma}$, where $\sigmabar(\tau) = \sum_{\sigma} m_{\sigma}(\tau) \cdot \sigma$ in the Grothendieck group of $\GL_2(k)$.
\end{thm}

We stress that this theorem of \cite{EGmoduli} is for all inertial types, in contrast to the Breuil--M\'ezard result of Theorem~\ref{thm:intro-main-Z}(2) which is only for tame types;\ in particular the cycles for Steinberg weights $\sigma$ do occur. In fact the theorem can be (and is) extended to cover  potentially semistable representations of Hodge type $0$ as well.

It remains to prove that the cycles $Z^{\sigma}$ are as in Theorem~\ref{thm:component-thm}, i.e., to check that $Z^{\sigma} = \cX^{\sigma}_{2,\red}$ when $\sigma$ is non-Steinberg, whereas $Z^{\chi \otimes \St} = \cX^{\chi}_{2,\red} + \cX^{\chi \otimes \St}_{2,\red}$. This is where the the results of the present paper enter.  We argue by transferring results from $\cZ^{\tau}$ to $\cX_2^{\tau,\BT}$ via a consideration of versal rings, without comparing the two stacks directly.\footnote{\emph{Added in revision}: it is now also possible to transfer these results using \cite[Thm~4.5]{APAW}.} In particular the ring $R^{\tau,\BT}_{\rbar}$ is a versal ring to $\cX^{\tau,\BT}$ at the point corresponding to $\rbar$;\ and so the formula $Z^{\sigma} = \cX^{\sigma}_{2,\red}$ in the non-Steinberg case will follow by an application of Theorem~\ref{thm:intro=gen-red} for a suitably chosen $\rbar$. The Steinberg case is handled directly using a semistable deformation ring. This completes the proof.

\section{Notation and conventions}\label{subsec: notation}

\subsubsection*{Galois theory}

 Let $p > 2$ be a prime number, 
 and fix  a finite extension $K/\Qp$, with
  residue field $k$ of cardinality $p^f$. In this paper we will study various stacks that are closely related to
the representation theory of $G_K$, the absolute Galois group of $K$.

Our representations of $G_K$ will have coefficients in $\Qpbar$,
a fixed algebraic closure of~$\Qp$ whose residue field we denote by~$\Fpbar$. Let $E$ be a finite
extension of $\Qp$ contained in $\Qpbar$. Write $\cO$ for the ring of integers in
$E$, with uniformiser $\varpi$ and residue field $\F \subset
\Fpbar$.  

As is often the case, we assume that our coefficients are ``sufficiently large''. Specifically, if $L$ is the quadratic unramified extension of $K$, we assume that $E$ admits an embedding of $K' = L(\pi^{1/(p^{2f}-1)})$ for some uniformizer $\pi$ of $K$. Write $l$ for the residue field of $L$.

Fix an embedding $\sigma_0:k \into\F$, and recursively define
$\sigma_i:k \into\F$ for all $i\in\Z$ so that
$\sigma_{i+1}^p=\sigma_i$. 
For each $i$ we
define the fundamental character $\omega_{\sigma_i}$ to be the composite \[\xymatrix{I_K \ar[r] & \cO_{K}^{\times}\ar[r] & k^{\times}\ar[r]^{\sigma_i} & \Fpbarx,}\]
where the map $I_K \to \cO_K^\times$ is induced by the restriction of the inverse of the Artin map, which we normalise so that
uniformisers correspond to geometric Frobenius elements.

\subsubsection*{Inertial local Langlands} A two-dimensional \emph{tame inertial type} is (the isomorphism
class of) a tamely ramified representation
$\tau : I_K \to \GL_2(\Zpbar)$ that extends to a representation of $G_K$ and
whose kernel is open.  
Such a representation is of the form $\tau
\simeq \eta \oplus \eta'$, and we say that $\tau$ is a \emph{tame principal series
  type} if 
$\eta,\eta'$ both extend to characters of $G_K$. Otherwise,
$\eta'=\eta^q$, and $\eta$ extends to a character of~$G_L$.
In this case we say
that~$\tau$ is a \emph{tame cuspidal type}. In either case $\tau|_{I_{K'}}$ is trivial, since $\tau$ is tame, and therefore a potentially crystalline representation of $G_K$ with inertial type~$\tau$ will become crystalline over $K'$.

Henniart's appendix to \cite{breuil-mezard}
associates a finite dimensional irreducible $E$-representation $\sigma(\tau)$ of
$\GL_2(\cO_K)$ to each inertial type $\tau$; we refer to this association as the {\em
  inertial local Langlands correspondence}. 
Since we are only working
with tame inertial types, this correspondence can be made very
explicit, as in \cegsCnotation. (Since we will not directly use the explicit description in this paper, we will not repeat it here.)

\subsubsection*{Serre weights and tame types}
By definition, a \emph{Serre weight} is an irreducible
$\F$-representation of $\GL_2(k)$. Then, concretely, a Serre weight
 is of the form
\[\sigmabar_{\vec{t},\vec{s}}:=\otimes^{f-1}_{j=0}
(\det{\!}^{t_j}\Sym^{s_j}k^2) \otimes_{k,\sigma_{j}} \F,\]
where $0\le s_j,t_j\le p-1$ and not all $t_j$ are equal to
$p-1$. We say that a Serre weight is \emph{Steinberg} if $s_j=p-1$ for all $j$,
and \emph{non-Steinberg} otherwise.

Let $\tau$ be a tame inertial type. 
 Write $\sigmabar(\tau)$ for the
semisimplification of the reduction modulo~$p$ of a
$\GL_2(\cO_K)$-stable $\cO$-lattice in $\sigma(\tau)$. The action
of~$\GL_2(\cO_K)$ on~$\sigmabar(\tau)$ factors through~$\GL_2(k)$, so
the Jordan--H\"older factors~$\JH(\sigmabar(\tau))$ of~$\sigmabar(\tau)$ are Serre weights.
By the results of~\cite{MR2392355}, these Jordan--H\"older factors of
$\sigmabar(\tau)$ are pairwise non-isomorphic, and are parametrised
by a certain  set $\cP_\tau$ 
that we now recall.

Suppose first that $\tau=\eta\oplus\eta'$ is a tame principal series 
type. Set $f'=f$ in this case.
We define $0 \le \gamma_i \le p-1$ (for $i \in \Z/f\Z$) 
to be the unique integers not all equal to $p-1$ such that 
$\eta (\eta')^{-1} = \prod_{i=0}^{f-1}
\omega_{\sigma_i}^{\gamma_i}$.

If instead $\tau = \eta \oplus \eta'$
is a cuspidal type, set $f'=2f$. We define $0 \le \gamma_i \le p-1$ (for $i \in \Z/f'\Z$) 
to be the unique integers such that 
$\eta (\eta')^{-1} = \prod_{i=0}^{f'-1}
\omega_{\sigma'_i}^{\gamma_i}$. Here 
$\sigma'_0 : l \to \Fpbar^{\times}$ is a fixed choice of embedding extending $\sigma_0$, 
$(\sigma'_{i+1})^p = \sigma'_i$ for all $i$, and the fundamental characters $\omega_{\sigma'_i} : I_L \to \Fpbar^{\times}$ for each $\sigma'_i : l \to \Fpbar^{\times}$ are defined in the same way as the $\omega_{\sigma_i}$.

If~$\tau$ is scalar then
we set $\cP_\tau=\{\varnothing\}$. 
Otherwise we have $\eta\ne\eta'$, and 
we let
$\cP_{\tau}$ be the collection of subsets $J \subset \Z/f'\Z$ 
satisfying the conditions:
\begin{itemize}
\item if $i-1\in J$ and $i\notin J$ then $\gamma_{i}\ne p-1$, and
\item if $i-1\notin J$ and $i\in J$ then $\gamma_{i}\ne 0$
\end{itemize}
and, in the cuspidal case, satisfying the further condition that $i
\in J$ if and only if $i+f \not\in J$.

The Jordan--H\"older factors of $\sigmabar(\tau)$ are by definition
Serre weights, and are
parametrised by $\cP_{\tau}$ as follows (see~\cite[\S3.2, 3.3]{emertongeesavitt}). For any $J\subseteq \Z/f'\Z$, we let $\delta_J$ denote
the characteristic function of $J$, and if $J \in \cP_{\tau}$
we define $s_{J,i}$ by
\[s_{J,i}=\begin{cases} p-1-\gamma_{i}-\delta_{J^c}(i)&\text{if }i-1 \in J \\
  \gamma_{i}-\delta_J(i)&\text{if }i-1\notin J, \end{cases}\]
and we set $t_{J,i}=\gamma_{i}+\delta_{J^c}(i)$ if $i-1\in J$ and $0$
otherwise. Write $\vec{s}$ for the tuple $(s_{J,i})_{0 \le i < f}$, suppressing the $J$ from the notation for readability, and similarly for $\vec{t}$.

In the principal series case we let
$\sigmabar(\tau)_J
:=\sigmabar_{\vec{t},\vec{s}}\otimes\eta'\circ\det$ for each $J \in \cP_{\tau}$;
the $\sigmabar(\tau)_J$ are precisely the Jordan--H\"older factors of
$\sigmabar(\tau)$. 

In the cuspidal case, one checks that $s_{J,i} = s_{J,i+f}$ for all
$i$, and also that  the character $\eta' \cdot
\prod_{i=0}^{f'-1} (\sigma'_i)^{t_{J,i}} : l^{\times} \to
\F^{\times}$ factors as $\theta \circ N_{l/k}$ where $N_{l/k}$
is the norm map. We let $\sigmabar(\tau)_J
:=\sigmabar_{0,\vec{s}}\otimes\theta \circ\det$, again for $J \in \cP_\tau$;
the $\sigmabar(\tau)_J$ are precisely the Jordan--H\"older factors of
$\sigmabar(\tau)$.

\subsubsection*{$p$-adic Hodge theory} We normalise Hodge--Tate weights so that all Hodge--Tate weights of
the cyclotomic character are equal to $-1$. We say that a potentially
crystalline representation $r :G_K\to\GL_2(\Qpbar)$ has \emph{Hodge
  type} $0$, or is \emph{potentially Barsotti--Tate}, if for each
$\varsigma :K\into \Qpbar$, the Hodge--Tate weights of $r$ with
respect to $\varsigma$ are $0$ and $1$.  (Note that this is a more
restrictive definition of potentially Barsotti--Tate than is sometimes
used; however, we will have no reason to deal with representations
with non-regular Hodge-Tate weights, and so we exclude them from
consideration. Note also that it is more usual in the literature to
say that $r$ is potentially Barsotti--Tate if it is potentially
crystalline, and $r^\vee$ has Hodge type $0$.) 

We say
that a potentially crystalline representation
$r:G_K\to\GL_2(\Qpbar)$ has  \emph{inertial type} $\tau$ if the traces of
elements of $I_K$ acting on~$\tau$ and on
\[\Dpcris(r)=\varinjlim_{K'/K}(\Bcris\otimes_{\Qp}V_r)^{G_{K'}}\] are
equal (here~$V_r$ is the underlying vector space
of~$V_r$). 
A representation $\rbar:G_K\to\GL_2(\Fpbar)$ \emph{has a potentially
    Barsotti--Tate lift of
    type~$\tau$} if and
  only if $\rbar$ admits a lift to a representation
  $r:G_K\to\GL_2(\Zpbar)$ of Hodge type~$0$ and inertial type~$\tau$.

\subsubsection*{Serre weights of mod~$p$ Galois representations}  Given a continuous representation $\rbar:G_K\to\GL_2(\Fpbar)$, there
is an associated (nonempty) set of Serre weights~$W(\rbar)$, defined 
to be the set of Serre weights $\sigmabar_{\vec{t},\vec{s}}$ such that $\rbar$ has a
crystalline lift whose Hodge--Tate weights are as follows:\ for each embedding $\sigma_j : k \into \F$ there is an
  embedding $\widetilde{\sigma}_j : K \into \Qpbar$ lifting $\sigma_j$
  such that the $\widetilde{\sigma}_j$-labelled Hodge--Tate weights
  of~$r$ are $\{-s_j-t_j,1-t_j\}$, and the remaining $(e-1)f$ pairs of  Hodge--Tate weights
  of~$r$ are all $\{0,1\}$.

There are in fact
several different definitions of~$W(\rbar)$ in the literature; as a
result of the papers~\cite{blggu2,geekisin,gls13}, these definitions
are known to be equivalent up to normalisation.  The normalisations in this paper are the same as those of \cite{cegsB,cegsC};\ see \cegsCnotation\ for a detailed discussion of these normalisations. In particular we have normalised the set of Serre weights so that $\rbar$ has a potentially Barsotti--Tate lift of type $\tau$ if and only if $W(\rbar) \cap \JH(\sigmabar(\tau)) \neq \varnothing$ (\cegsBtresramlemma).

\subsubsection*{Stacks}We follow the terminology of~\cite{stacks-project}; in
particular, we write ``algebraic stack'' rather than ``Artin stack''. More
precisely, an algebraic stack is a stack in groupoids in the \emph{fppf} topology,
whose diagonal is representable by algebraic spaces, which admits a smooth
surjection from a
scheme. See~\cite[\href{http://stacks.math.columbia.edu/tag/026N}{Tag
  026N}]{stacks-project} for a discussion of how this definition relates to
others in the literature, and~\cite[\href{http://stacks.math.columbia.edu/tag/04XB}{Tag
  04XB}]{stacks-project} for key properties of morphisms
representable by algebraic spaces.

For a commutative ring $A$, an \emph{fppf stack over $A$} (or
\emph{fppf} $A$-stack) is a stack fibred in groupoids over the big \emph{fppf}
site of $\Spec A$. Following \cite[Defs.~5.3, 7.6]{Emertonformalstacks}, an {\em fppf} stack in groupoids $\cX$ over a scheme $S$ is called
a {\em formal algebraic stack} if
there is a morphism $U \to \cX$,
  whose domain $U$ is a formal algebraic space over $S$
  (in the sense 
of~\cite[\href{http://stacks.math.columbia.edu/tag/0AIL}{Tag 0AIL}]{stacks-project}),
and which is representable by algebraic spaces,
smooth, and surjective.

Let $\Spf \cO$ denote the affine formal
scheme (or affine formal algebraic space, in the terminology
of~\cite{stacks-project}) obtained by $\varpi$-adically completing
$\Spec \cO$.    A formal algebraic stack $\cX$ over $\Spec \cO$
  is called $\varpi$-adic if the canonical map
  $\cX \to \Spec \cO$ factors 
  through $\Spf \cO$,
  and if the induced map $\cX \to \Spf \cO$
  is algebraic, i.e.\ representable by algebraic
  stacks (in the sense 
of~\cite[\href{http://stacks.math.columbia.edu/tag/06CF}{Tag 06CF}]{stacks-project} and \cite[Def.~3.1]{Emertonformalstacks}).

\section{Moduli stacks of Breuil--Kisin modules and \'etale \texorpdfstring{$\varphi$}{phi}-modules}
\label{sec:recollections}

The main object of study in this paper is the $\varpi$-adic formal algebraic stack $\cZ^{\dd}$ that
was introduced and studied in
\cite{cegsB,cegsC}, and whose $\Fpbar$-points are naturally in bijection with the continuous representations $\rbar : G_K \to \GL_2(\Fpbar)$ admitting a potentially Barsotti--Tate lift of some tame type. In this section we review the construction and known properties of $\cZ^{\dd}$, as well as those of several other closely related stacks.

\subsubsection*{Stacks of Breuil--Kisin modules} For each tame type $\tau$, there is a $\varpi$-adic formal algebraic stack $\cC^{\tau,\BT}$ whose $\Spf(\cO_{E'})$-points, for any finite extension $E'/E$, are the Breuil--Kisin modules correpsonding to $2$-dimensional potentially Barsotti--Tate representations of type $\tau$.  This stack is constructed in several steps, which we review in brief. (We refer the reader to \cite{cegsB} as well as to the summary in \cegsCrecollections\ for complete definitions, recalling here only what will be used in this paper.)

For each integer $a \ge 1$, we write $\cC^{\dd,a}$ for  the {\em fppf} stack over~$\cO/\varpi^a$ which associates to any $\cO/\varpi^a$-algebra $A$
the groupoid $\cC^{\dd,a}(A)$ 
of rank~$2$ Breuil--Kisin modules of height at most~$1$ with $A$-coefficients and descent data from
$K'$ to~$K$. Set $\cC^{\dd} = \dirlim_a \cC^{\dd,a}$. The closed substack $\cC^{\dd,\BT}$ of $\cC^{\dd}$ is cut out by a Kottwitz-type determinant condition, which can be thought of (on the Galois side) as cutting out the tamely potentially Barsotti--Tate representations from among all tamely potentially crystalline representations with Hodge--Tate weights in $\{0,1\}$. By \cegsBdisjointunion\ the stack $\cC^{\dd,\BT}$ then decomposes as a disjoint union of closed substacks $\cC^{\tau,\BT}$, one for each tame type $\tau$, consisting of Breuil--Kisin modules with descent data of type~$\tau$. Finally, for each $a\ge 1$ we write $\cC^{\tau,\BT,a} = \cC^{\tau,\BT} \times_{\cO} \cO/\varpi^a$, and similarly for $\cC^{\dd,\BT,a}$. The following properties are established in \cegsBCmainresultsplus.

\begin{thm} The stacks $\cC^{\dd,a}$, $\cC^{\dd,\BT,a}$, and $\cC^{\tau,\BT,a}$ are algebraic stacks of finite type over $\cO$, while the stacks $\cC^{\dd}$, $\cC^{\dd,\BT}$, and $\cC^{\tau,\BT}$ are $\varpi$-adic formal algebraic stacks. Moreover:
\begin{enumerate}
  \item $\cC^{\tau,\BT}$ is analytically normal, Cohen--Macaulay, and flat over $\cO$.
  \item The stacks $\cC^{\dd,\BT,a}$ and $\cC^{\tau,\BT,a}$ are equidimensional of dimension $[K:\Qp]$.
  \item The special fibres $\cC^{\dd,\BT,1}$ and $\cC^{\tau,\BT,1}$ are reduced.

\end{enumerate}
\end{thm}

\subsubsection*{Galois moduli stacks} Let 
 $\cR^{\dd,a}_{K'}$ be the \emph{fppf} $\F$-stack which
  associates
  to any $\cO/\varpi^a$-algebra $A$ the groupoid $\cR^{\dd,a}_{K'}(A)$ of rank $2$ \'etale
  $\varphi$-modules  with $A$-coefficients and  descent data from $K'$ to
  $K$.  We will usually suppress $K'$ from the notation. Inverting $u$ on Breuil--Kisin modules gives a proper morphism $\cC^{\dd,a} \to \cR^{\dd,a}$, which
then restricts to proper morphisms $\cC^{\dd,\BT,a} \to \cR^{\dd,a}$ as well as 
$\cC^{\tau,\BT,a} \to \cR^{\dd,a}$ for each
$\tau$.

The paper \cite{EGstacktheoreticimages} develops a theory of scheme-theoretic images of proper morphisms $\cX \to \cF$  of stacks over a locally
Noetherian base-scheme~$S$, where $\cX$ is an algebraic stack which is locally of finite presentation over~$S$,
and the diagonal of $\cF$ is representable by algebraic spaces and locally of
finite presentation. This theory applies in particular to each of the morphisms of the previous paragraph (even though $\cR^{\dd,a}$ is \emph{not} algebraic).
We define $\cZ^{\dd,a}$ and $\cZ^{\tau,a}$ to be the scheme-theoretic images of the morphisms  $\cC^{\dd,\BT,a} \to \cR^{\dd,a}$ and $\cC^{\tau,\BT,a} \to \cR^{\dd,a}$, respectively.
Set $\cZ^{\dd} = \dirlim_{a} \cZ^{\dd,a}$ and $\cZ^{\tau} = \dirlim_{a} \cZ^{\tau,a}$. The following theorem combines \cegsBZmainresultsplus.

\begin{thm}\label{lem: C 1 and Z 1 are the underlying reduced substacks}
The stacks $\cZ^{\dd,a}$ and $\cZ^{\tau,a}$ are algebraic stacks of finite type over~$\cO$, while the stacks $\cZ^{\dd}$ and $\cZ^{\tau}$ are $\varpi$-adic formal algebraic stacks. Moreover:
\begin{enumerate}

  \item The  stacks $\cZ^{\dd,a}$ and  $\cZ^{\tau,a}$ are equidimensional of
dimension~$[K:\Qp]$. 
  \item The stacks $\cZ^{\dd,1}$ and $\cZ^{\tau,1}$ are reduced.

    \item The $\Fpbar$-points of $\cZ^{\dd,1}$ are naturally in bijection with the
    continuous representations $\rbar:G_K\to\GL_2(\Fpbar)$ which are not a
    twist of a tr\`es ramifi\'ee extension of the trivial character by the
    mod~$p$ cyclotomic character. Similarly, the $\Fpbar$-points of $\cZ^{\tau,1}$ are naturally in bijection with the
    continuous representations $\rbar:G_K\to\GL_2(\Fpbar)$ which have
    a potentially Barsotti--Tate lift of type~$\tau$.

\end{enumerate}
In particular the stack $\cZ^{\dd,1}$ is the underlying reduced substack of $\cZ^{\dd,a}$ for each $a \ge 1$, as well as of $\cZ^{\dd}$; and similarly for the stacks $\cZ^{\tau,1}$.
\end{thm}

\begin{remark} We stress that the morphism $\cZ^{\dd,a} \into \cZ^{\dd} \times_{\cO} \cO/\varpi^a$ need not be an isomorphism \emph{a priori}, and we have no reason to expect that it is. However, our results in the next section will prove that it is \emph{generically} an isomorphism for all $a \ge 1$.
\end{remark}

\subsubsection*{Versal rings and deformation rings}
Let $x$ be an $\F'$-point of $\cZ^{\tau,a}$, corresponding to the representation $\rbar : G_K \to \GL_2(\F')$. We will usually write  $R_{\rbar}^{\tau,\BT}$ for the reduced and $p$-torsion free  quotient of the universal framed deformation ring of $\rbar$ whose $\Qpbar$-points correspond
to the potentially Barsotti--Tate lifts of~$\rbar$ of
type~$\tau$. (In Section~\ref{sec:appendix on geom BM} we will denote this ring instead by $R_{\rbar,0,\tau}$, for ease of comparison with the paper \cite{geekisin}.)

It is explained in \cegsBlastsection\ that there are versal rings $R^{\tau,a}_x$ to $\cZ^{\tau,a}$ at the point $x$, such that the following holds. (These rings are denoted $R^{\tau,a}$ in \cite{cegsB};\ we include the subscript~$x$ here to emphasize the dependence on the point $x$.)

\begin{prop}[\cegsBversalprop]
  \label{cor: R tau BT is a versal ring to Z-hat} We have $\varprojlim
  R^{\tau,a}_x=R^{\tau,\BT}_{\rbar}$; thus $R^{\tau,\BT}_{\rbar}$ is a
  versal ring to ${\cZ^\tau}$ at~$x$.
\end{prop}

Similarly there is a versal ring $R^{\dd,a}_x$ to $\cZ^{\dd,a}$ at $x$, and each $R^{\tau,a}_x$ is a quotient of $R^{\dd,a}_x$.

\subsubsection*{Irreducible components of $\cC^{\tau,\BT,1}$ and $\cZ^{\tau,1}$}
Fix a tame type $\tau$, and recall that we set $f' = f$ if the type $\tau$ is principal series, while $f' = 2f$ if the type $\tau$ is cuspidal. We say that a subset $J \subset \Z/f'\Z$ is a \emph{profile} if:
\begin{itemize}
  \item $\tau$ is scalar and $J = \varnothing$, 
  \item $\tau$ is a non-scalar principal series type and $J$ is arbitrary, or
  \item $\tau$ is cupsidal and $J$ has the property that $i \in J$ if and only if $i+f \not\in J$.
\end{itemize}
 Thus there are exactly $2^f$ profiles if $\tau$ is non-scalar. The set $\cP_{\tau}$ introduced in Section~\ref{subsec: notation} is a subset of the set of profiles.

To each profile $J$, the discussion in \cegsCextensionssec\ associates a closed substack $\overline{\mathcal{C}}(J)$ of $\cC^{\tau,\BT,1}$. The stack $\overline{\cZ}(J)$ is then defined to be the scheme-theoretic image of $\overline{\cC}(J)$ under the  map $\cC^{\tau,\BT,1} \to \cZ^{\tau,1}$.

The following description of the irreducible components of $\cC^{\tau,\BT,1}$ and $\cZ^{\tau,1}$ is proved in \cegsCcomponentdescription;\  the description of the components of $\cZ^{\tau,1}$ in part (3) of the theorem is analogous to the description of the components of $\cX_{2,\red}$ of Theorem~\ref{thm:intro-thm-EG}.

\begin{thm}\label{thm:main thm cegsC} The irreducible components of $\cC^{\tau,\BT,1}$ and $\cZ^{\tau,1}$ are as follows. 
\begin{enumerate} 
  \item  The irreducible
    components of~$\cC^{\tau,1}$ are precisely the~$\overline{\cC}(J)$
    for profiles~$J$, and if $J\ne J'$ then~$\overline{\cC}(J)\ne\overline{\cC}(J')$.

 \item The irreducible
    components of~$\cZ^{\tau,1}$ are precisely the~$\overline{\cZ}(J)$
    for profiles~$J\in\cP_\tau$, and if $J\ne J'$ then~$\overline{\cZ}(J)\ne\overline{\cZ}(J')$.

  \item For each $J \in \cP_{\tau}$ there is a dense open substack $\cU$ of 
$\overline{\cC}(J)$ such that the  map $\overline{\cC}(J) 
\to \overline{\cZ}(J)$ restricts to an open immersion on $\cU$.

 \item 
  For each $J\in\cP_\tau$, there is a dense set of finite type
  points of $\overline{\cZ}(J)$ with the property that the corresponding Galois
  representations have $\sigmabar(\tau)_J$ as a Serre weight, and which
  furthermore admit a unique Breuil--Kisin model of type~$\tau$ as defined in \cegsCmodeldefn.
{}
\end{enumerate}
\end{thm}

\begin{remark}\label{rem:maximalfamilies} If $\sigmabar(\tau)_J = \sigmabar_{\vec{t},\vec{s}}$, then the dense set of finite type points of $\overline{\cZ}(J)$ produced in the proof of \cegsCbigfamilies, as claimed in 
Theorem~\ref{thm:main thm cegsC}(4), consists of points corresponding to reducible representations $\rbar$ such that $\rbar|_{I_K}$ is an extension of $\overline{\varepsilon}^{-1} \prod_{i=0}^{f-1} \omega_{\sigma_i}^{t_i}$ by $\prod_{i=0}^{f-1} \omega_{\sigma_i}^{s_i + t_i}$ (necessarily peu ramifi\'e in case the ratio of the two characters is cyclotomic).
\end{remark}

\begin{remark}\label{rem:verticalrem}
We emphasize in Theorem~\ref{thm:main thm cegsC} that the components of $\cZ^{\tau,1}$ are indexed by profiles $J \in \cP_{\tau}$, \emph{not} by all profiles. If $J \not\in \cP_{\tau}$, then by \cegsCverticalcomponents\ the stack $\overline{\cZ}(J)$ has dimension strictly smaller than $[K:\Qp]$, and so is properly contained in some component of $\cZ^{\tau,1}$.
\end{remark}

\begin{remark}\label{rem:dd-field} Strictly speaking, in the principal series case Theorem~\ref{thm:main thm cegsC} is proved in \cite{cegsC} for stacks of Breuil--Kisin modules and of \'etale $\varphi$-modules with descent data from $K(\pi^{1/(p^f-1)})$ to $K$, rather than from our $K'$ to $K$. 
But in fact we can replace $K(\pi^{1/(p^f-1)})$ with any extension of prime-to-$p$ degree that remains Galois over $K$, such as $K'$, without changing the resulting stacks.  This follows from \cegsBfieldchange, together with the isomorphism $\cR^{\dd}_{K(\pi^{1/(p^f-1)})} \toisom \cR^{\dd}_{K'}$ of stacks of \'etale $\varphi$-modules (which is easier  than \emph{loc. cit.}, since multiplication by $u$ on an \'etale $\varphi$-module is bijective, and can be proved by the same argument as in \cite[Cor.~2.3.21]{EGmoduli}).
\end{remark}

\section{Generic reducedness of \texorpdfstring{$\Spec R^{\tau,\BT}_{\rbar}/\varpi$}{potentially Barsotti--Tate
    deformation ring}}
\label{subsec: generically
  reduced}

Fix a Galois representation $\rbar:G_K\to\GL_2(\F')$, where $\F'/\F$
is a finite extension. Our goal in this section is to prove that the scheme $\Spec R^{\tau,\BT}_{\rbar}/\varpi$ is generically reduced;\ this will be a key ingredient in the proof of our geometric
Breuil--M\'ezard result in Section~\ref{sec: picture}.  Recall that a scheme is
generically reduced if it contains an open reduced subscheme whose
underlying topological space is dense;\ in the case of a Noetherian
affine scheme $\Spec A$, this is equivalent to requiring
that the localisation of $A$ at each of its minimal
primes is reduced.

We may of course suppose that $R^{\tau,\BT}_{\rbar} \neq 0$, so that $\rbar$ has a potentially Barsotti--Tate lift of type $\tau$, and so corresponds to a finite type point $x : \Spec\F'\to\cZ^{\tau,a}$ for any $a \ge 1$. 
It follows from Proposition~\ref{cor: R tau BT is a versal
  ring to Z-hat} that $\Spec R^{\tau,a}_x$ is a closed subscheme of
$\Spec R^{\tau,\BT}_{\rbar}/\varpi^a$, but we have no reason to
believe that equality holds.
It follows from Theorem~\ref{lem: C 1 and Z 1 are the underlying reduced substacks}(2),
together with Lemma~\ref{lem: generic reducedness passes to completions} below, that
$\Spec R^{\tau,1}_x$ is the underlying reduced subscheme
of~$\Spec R^{\tau,\BT}_{\rbar}/\varpi$, so that equality holds in the case $a = 1$
if and only
if~$\Spec R^{\tau,\BT}_{\rbar}/\varpi$ is reduced;\ however, again, we have no
reason to believe that this holds in general. 

Nevertheless it \emph{is} the case that 
$\Spec R^{\tau,\BT}_{\rbar}/\varpi$  is generically
reduced;\ \emph{cf.}\ Theorem~\ref{prop: generically reduced special fibre deformation ring} below. We will deduce
 Theorem~\ref{prop: generically reduced special fibre deformation ring}
 from the following global statement.

  \begin{prop}
    \label{prop: existence of dense open substack of R such that C is a mono}
Let~$\tau$ be a tame type. There is a dense open substack $\cU$ of $\cZ^{\tau}$
      such that $\cU_{/\F}$ is reduced.
  \end{prop}
  \begin{proof}
	  The proposition will follow from an application of
	  Lemma~\ref{prop:opens},
	  and the key to this application will be to find a
	  candidate open substack $\cU^1$ of $\cZ^{\tau,1}$,
which we will do using our study of the irreducible components
of $\cC^{\tau,\BT,1}$ and $\cZ^{\tau,1}$.  

  Recall that, for each profile $J \in \cP_\tau$,
 we let $\overline{\cZ}(J)$ denote the scheme-theoretic image of
 $\overline{\cC}(J)$ under the proper morphism
 $\cC^{\tau,\BT,1} \to \cZ^{\tau,1}$. 
 Each $\overline{\cZ}(J)$ 
 is a closed substack of $\cZ^{\tau,1}$, and so,
  if we let $\cV(J)$ be the complement in $\cZ^{\tau,1}$ of the
  union of the $\overline{\cZ}(J')$ for all profiles $J'\ne J$ then $\cV(J)$ is
  a dense open substack of~$\overline{\cZ}(J)$, by Theorem~\ref{thm:main thm cegsC}(2) and Remark~\ref{rem:verticalrem} (the former in consideration of $J' \in \cP_{\tau}$, the latter for $J' \not\in \cP_{\tau}$).
The preimage $\cW(J)$ of $\cV(J)$ in $\cC^{\tau,\BT,1}$ is therefore a
dense open
substack of $\overline{\cC}(J)$. 
Possibly shrinking $\cW(J)$ further, we may
  suppose by Theorem~\ref{thm:main thm cegsC}(3) 
that the morphism $\cW(J)\to
\cZ^{\tau,1}$ is an open immersion. 

Write $| \cdot |$ for the underlying topological space of a stack.
The complement $|\overline{\cC}(J)|\setminus |\cW(J)|$ is a closed subset
of $|\overline{\cC}(J)|$, and thus of
$|\cC^{\tau,\BT,1}|$, and its image under the proper morphism
$\cC^{\tau,\BT,1}\to \cZ^{\tau,1}$ is a closed
subset of $|\cZ^{\tau,1}|$, which is (e.g.\ for dimension reasons\footnote{Choose,
as we may,
a surjective smooth morphism $U \to \overline{\cZ}(J)$ with $U$ a finite type $\F$-scheme.
Let $V := \overline{\cC}(J)\times_{\overline{\cZ}(J)} U$. Then the projection
$V \to \overline{\cC}(J)$
is again smooth and surjective, while the projection
$f:V \to U$ is representable by schemes and proper; 
in particular, $V$ is also 
a finite type $\F$-scheme. 
Let $W:= \cW(J) \times_{\overline{\cZ}(J)} U.$  Then $W$ is
a dense open subscheme of~$U$, equipped with a section $W \to V$ which realizes it
as a dense open subscheme of $V$ as well. 
Write $Y := V\setminus W$ (a closed subset of~$V$) and $X = f(Y)$
(a closed subset of~$U$). 
If $T$ is any irreducible component of~$V$, 
then $\dim T = \dim (T\cap W) = \dim \bigl(f(T) \cap W\bigr) = \dim f(T)$
(in particular, $f(T)$ is an irreducible component of~$U$; and each irreducible
component of $U$ arises in this manner),
while $\dim (T\cap Y) < \dim T.$ 
Thus also $\dim f(T\cap Y) <  \dim T = \dim f(T)$, so that $f(T\cap Y)$ is a proper
closed subset of~$f(T)$.  Since $Y = \bigcup_T (T\cap Y),$
we see that $X = \bigcup_T f(T\cap Y)$ is a proper closed subset of~$U$,   
whose complement $U'$ is dense in~$U$.
The image of $U'$ is then a dense (equivalently, non-empty) open 
substack of $\overline{\cZ}(J)$; and this image coincides with 
the complement in~$\overline{\cZ}(J)$ of the image of $\overline{\cC}(J) \setminus \cW(J)$
(since $U'$ is the preimage of this complement, by construction).})
a proper closed subset of~$|\overline{\cZ}(J)|$; so if we let
$\cU(J)$ be the complement in $\cV(J)$ of this image, then $\cU(J)$ is
open and dense in~$\overline{\cZ}(J)$, 
and the morphism $\cC^{\tau,\BT,1}\times_{\cZ^{\tau,1}}
\cU(J)\to \cU(J)$ is a monomorphism. Set $\cU^1=\cup_J \cU(J)$.
Since the $\cU(J)$ are pairwise disjoint
by construction, 
\numequation\label{eq:monomorphism1}
\cC^{\tau,\BT,1}\times_{\cZ^{\tau,1}}
\cU^1\to \cU^1\end{equation} is again a monomorphism. By construction (taking into account
Theorem~\ref{thm:main thm cegsC}(2)), $\cU^1$ is dense in $\cZ^{\tau,1}$.

Now let $\cU$ denote the open substack of $\cZ^{\tau}$ corresponding
to $\cU^1$.  Since $|\cZ^{\tau}| = |\cZ^{\tau,1}|$,
we see that $\cU$ is dense in $\cZ^\tau$. We have seen in the previous
paragraph that the statement of Lemma~\ref{prop:opens}~(5) holds
(taking~$a=1$, $\cX=\cC^{\tau,\BT}$, and~$\cY=\cZ^{\tau}$); so
Lemma~\ref{prop:opens} implies that, for each $a \geq 1$,
the closed immersion
      \numequation\label{eq:monomorphism2} \cU\times_{\cZ^{\tau}}\cZ^{\tau,a} \hookrightarrow \cU\times_{\cO} \cO/\varpi^a \end{equation}
      is an isomorphism.

      In particular,
      since the closed immersion $\cU^1 = \cU\times_{\cZ^{\tau}} \cZ^{\tau,1}
      \to \cU_{/\F}$ is an isomorphism, 
      we may regard $\cU_{/\F}$ as an open substack
      of $\cZ^{\tau,1}$.  Since $\cZ^{\tau,1}$ is
      reduced,
      by Theorem~\ref{lem: C 1 and Z 1 are the underlying reduced substacks}(2),
    so is its open substack $\cU_{/\F}$.
    This completes the proof of the proposition.
  \end{proof}

\begin{cor}
    \label{cor: existence of dense open substack of Z with unique
      Kisin module}
Let~$\tau$ be a tame type. There is a dense open substack $\cU$ of
$\cZ^{\tau}$ such that we have an isomorphism
	$\cC^{\tau,\BT} \times_{\cZ^{\tau}} \cU \iso \cU,$
	as well as isomorphisms
	$$\cU\times_{\cZ^{\tau}} \cC^{\tau,\BT, a} \iso 
	\cU\times_{\cZ^{\tau}}\cZ^{\tau,a}
	\iso \cU\times_{\cO} \cO/\varpi^a,$$
	for each $a \geq 1$.
  \end{cor}
  \begin{proof}Taking $\cU$ as constructed in the proof of Proposition~\ref{prop: existence of dense open substack of R such that C is a mono}, this follows from Proposition~\ref{prop:opens} applied to the monomorphism \eqref{eq:monomorphism1} and isomorphism \eqref{eq:monomorphism2}.
  \end{proof}

\begin{remark}
	\label{rem:U^a = U mod pi^a} 
More colloquially, Corollary~\ref{cor: existence of dense open substack of Z with unique
      Kisin module} shows that for each tame type~$\tau$, there is an
    open dense substack~$\cU$ of~$\cZ^\tau$ consisting of Galois
    representations which have a unique Breuil--Kisin
    model of type~$\tau$. 
\end{remark}

  \begin{lemma}
    \label{lem: flat pullbacks are reduced scheme theoretically
      dense} 
      If $\cU$ is an open substack of $\cZ^{\tau}$ satisfying the condition
      of Proposition~{\em \ref{prop: existence of dense open substack of R such that C is a mono}},
      and if $T \to \cZ^{\tau}_{/\F}$ is a smooth morphism
      whose source is a scheme, 
      then $T\times_{\cZ^{\tau}_{/\F}} \cU_{/\F}$ is reduced,
      and is a dense open subscheme of $T$.
  \end{lemma}
  \begin{proof}
Since $\cZ^{\tau}_{/\F}$ is a 
    Noetherian algebraic stack (being of finite presentation over $\Spec \F$),
    the open
    immersion \[\cU_{/\F}\to \cZ^{\tau}_{/\F}\]
    is quasi-compact (\cite[\href{https://stacks.math.columbia.edu/tag/0CPM}{Tag 0CPM}]{stacks-project}), and has dense image by assumption. 
Again by assumption, it factors through $\cZ^{\tau}_{\red}$ ($= (\cZ^{\tau}_{/\F})_{\red}$),
and the resulting open 
    immersion \[\cU_{/\F}\to \cZ^{\tau}_{\red}\]
is again quasi-compact with dense image.  Since its target is reduced,
it is in fact scheme-theoretically dominant.

Now the given smooth morphism $T\to \cZ^{\tau}_{/\F}$
base-changes to a smooth morphism
$$ T\times_{\cZ^{\tau}_{/\F}} \cZ^{\tau}_{\red} \to
\cZ^{\tau}_{\red},$$ whose source is equal to the underlying reduced
scheme~$T_{\red}$ of~$T$. (Indeed the source is reduced, because
property of being reduced is local for the smooth topology,
\cite[\href{https://stacks.math.columbia.edu/tag/04YH}{Tag
  04YH}]{stacks-project}; and it is a closed subscheme of~$T$ with
underlying topological space equal to that of~$T$, by \cite[\href{https://stacks.math.columbia.edu/tag/04XH}{Tag 04XH}]{stacks-project}.)
Since smooth morphisms are in particular flat, 
the pullback
$$T\times_{\cZ^{\tau}_{/\F}} \cU_{/\F}  = T_{\red}\times_{\cZ^{\tau}_{\red}} \cU_{/\F} \to T_{\red}$$ is an open immersion
    with dense image;\ here we use the fact that for a quasi-compact
	    morphism, the property of being scheme-theoretically dominant
	    is preserved by flat base-change.
    Since the source of
this pullback is open in $T_{\red}$, it is itself reduced.
  \end{proof}

The following result is standard, but we recall the proof for the sake of completeness.

  \begin{lem}
    \label{lem: generic reducedness passes to completions}
Let $T$ be a Noetherian scheme, all of whose local rings at
    finite type points are $G$-rings. If $T$ is reduced {\em (}resp.\ generically reduced{\em)},
    then so are all of its complete local rings at finite type points.
  \end{lem}
  \begin{proof}
    Let $t$ be a finite type point of~$T$, and write
    $A:=\cO_{T,t}$. Then $A$ is a (generically) reduced local $G$-ring,
    and we need to show that its completion $\widehat{A}$ is also
    (generically) reduced. Let $\widehat{\p}$ be a (minimal) prime of
    $\widehat{A}$; since $A\to\widehat{A}$ is (faithfully) flat,
    $\widehat{\p}$ lies over a (minimal) prime $\p$ of $A$ by the
    going-down theorem.

Then $A_\p$ is reduced by assumption, and we need to show that
$\widehat{A}_{\widehat{\p}}$ is
reduced. By~\cite[\href{http://stacks.math.columbia.edu/tag/07QK}{Tag
  07QK}]{stacks-project}, it is enough to show that the morphism $A\to
\widehat{A}_{\widehat{\p}}$ is regular. Both $A$ and $\widehat{A}$ are
$G$-rings (the latter
by~\cite[\href{http://stacks.math.columbia.edu/tag/07PS}{Tag
  07PS}]{stacks-project}), so the composite \[A\to \widehat{A}\to
(\widehat{A}_{\widehat{\p}})^{\widehat{}}\] is a composite of regular
morphisms, and is thus a regular morphism
by~\cite[\href{http://stacks.math.columbia.edu/tag/07QI}{Tag
  07QI}]{stacks-project}. 

This composite factors through the natural morphism $A_{\p}\to
(\widehat{A}_{\widehat{\p}})^{\widehat{}}$, so this morphism is also
regular. Factoring it as the composite \[A_{\p}\to \widehat{A}_{\widehat{\p}}\to
(\widehat{A}_{\widehat{\p}})^{\widehat{}},\]it follows
from~\cite[\href{http://stacks.math.columbia.edu/tag/07NT}{Tag
  07NT}]{stacks-project} that $A_{\p}\to \widehat{A}_{\widehat{\p}}$ is
regular, as required.
  \end{proof}

Finally, we are ready to prove the main result of this section.

\begin{thm}
  \label{prop: generically reduced special fibre deformation ring}
For any tame type~$\tau$, the scheme $\Spec R^{\tau,\BT}_{\rbar}/\varpi$ is generically reduced, with
  underlying reduced subscheme $\Spec R^{\tau,1}_x$. 
\end{thm}

  \begin{proof} By Proposition~\ref{cor: R tau BT is a versal ring to Z-hat},
  we have a versal morphism \[\Spf
  R_{\rbar}^{\tau,\BT}/\varpi\to \cZ^\tau_{/\F}\] 
at the point of $\cZ^\tau_{/\F}$ corresponding to $\rbar : G_K\to \GL_2(\F')$.
Since $\cZ^{\tau}_{/\F}$ is an algebraic stack of finite presentation
over $\F$ (as $\cZ^{\tau}$ is a
$\varpi$-adic formal algebraic stack of finite presentation over $\Spf \cO$), 
we may apply~\cite[\href{https://stacks.math.columbia.edu/tag/0DR0}{Tag 0DR0}]{stacks-project} to
this morphism so as to find a
smooth morphism $V\to \cZ^{\tau}_{/\F}$ with source a finite
type $\cO/\varpi$-scheme, and a point $v\in V$ with residue
field~$\F'$, such that there is an isomorphism
$\widehat{\cO}_{V,v}\cong R^{\tau,\BT}_{\rbar}/\varpi$,
compatible with the given morphism to~$\cZ^{\tau}_{/\F}$.
Proposition~\ref{prop: existence of dense open substack of R such that C is a mono}
and Lemma~\ref{lem: flat pullbacks are reduced scheme theoretically
      dense} taken together show that $V$ is generically reduced, and so the result follows from Lemma~\ref{lem: generic reducedness passes to completions}.
\end{proof}

\section{A case of the classical geometric Breuil--M\'ezard Conjecture}\label{sec:appendix on geom BM} In this
section, by combining the methods of~\cite{emertongeerefinedBM}
and~\cite{geekisin} we prove a special case of the classical geometric Breuil--M\'ezard
conjecture~\cite[Conj.\ 4.2.1]{emertongeerefinedBM}. This result is
``globalised'' in Section~\ref{sec: picture}.

Let
$\rbar:G_K\to\GL_2(\F)$ be a continuous representation, and let
$R^\square_{\rbar}$ be the universal framed deformation $\cO$-algebra for~$\rbar$. 
In this section we write $R^\square_{\rbar,0,\tau}$ for the quotient of $R^\square_{\rbar}$
that elswhere we have denoted 
$R^{\tau,\BT}_{\rbar}$. We use the more cumbersome
notation~$R^\square_{\rbar,0,\tau}$ here to make it easier for the
reader to refer to~\cite{emertongeerefinedBM} and~\cite{geekisin}. 

By \cite[Prop.\ 4.1.2]{emertongeerefinedBM}, $R^\square_{\rbar,0,\tau}/\varpi$ is zero
if $\rbar$ has no potentially Barsotti--Tate lifts of type~$\tau$, and otherwise
it is equidimensional of dimension $4+[K:\Qp]$. Each $\Spec
R^\square_{\rbar,0,\tau}/\varpi$ is a closed subscheme of $\Spec
R^\square_{\rbar}/\varpi$, and we write $Z(R^\square_{\rbar,0,\tau}/\varpi)$ for the
corresponding cycle, as in~\cite[Defn.\ 2.2.5]{emertongeerefinedBM}. This
is a formal sum of the irreducible components of $\Spec
R^\square_{\rbar,0,\tau}/\varpi$, weighted by the multiplicities with which they occur.

\begin{lem}
  \label{lem:write Serre weight as linear combination of types}If $\sigma$ is a
non-Steinberg Serre weight of $\GL_2(k)$, then there are
integers $n_\tau(\sigma)$ such that
$\sigma=\sum_\tau n_\tau(\sigma)\sigmabar(\tau)$ in the Grothendieck group
of mod $p$ representations of $\GL_2(k)$, where the $\tau$ run over the tame types.
\end{lem}
\begin{proof}This is an immediate consequence of the surjectivity of the natural
  map from the Grothendieck group of $\Qpbar$-representations of $\GL_2(k)$ to
  the Grothendieck group of $\Fpbar$-representations of $\GL_2(k)$ \cite[\S III,
  Thm.\ 33]{MR0450380}, together with the observation that the reduction of the
  Steinberg representation of $\GL_2(k)$ is precisely $\sigmabar_{\vec{0},\vec{p-1}}$.
  \end{proof}
Let $\sigma$ be a non-Steinberg Serre weight of $\GL_2(k)$, so that by
Lemma~\ref{lem:write Serre weight as linear combination of types} we can
write \numequation\label{eqn: defn of n_tau(sigmabar)}\sigma=\sum_\tau n_\tau(\sigma)\sigmabar(\tau)\end{equation}  in the Grothendieck group
of mod $p$ representations of $\GL_2(k)$. Note that the integers $n_\tau(\sigma)$ are not uniquely determined; however,
all our constructions elsewhere in this paper will be (non-obviously!) independent of the choice of the $n_\tau(\sigma)$. We also
write \[\sigmabar(\tau)=\sum_{\sigma}m_{\sigma}(\tau)\sigma;\] since
$\sigmabar(\tau)$ is multiplicity-free, each $m_{\sigma}(\tau)$ is equal to
$0$ or $1$. Then\[\sigma=\sum_{\sigma'}\left(\sum_\tau
  n_\tau(\sigma)m_{\sigma'}(\tau)\right)\sigma',\]and
therefore\numequation\label{eqn: orthogonality of sigmabar tau}\sum_\tau
  n_\tau(\sigma)m_{\sigma'}(\tau)=\delta_{\sigma,\sigma'}.\end{equation}

For each non-Steinberg Serre weight $\sigma$, we
set \[\cC_{\sigma}:=\sum_\tau
n_\tau(\sigma)Z(R^\square_{\rbar,0,\tau}/\varpi),\]where the sum ranges over the
tame types $\tau$, and the integers $n_\tau(\sigma)$ are as in~(\ref{eqn:
  defn of n_tau(sigmabar)}). By definition this is a formal sum with (possibly
negative) multiplicities of irreducible
subschemes of~$\Spec
R^\square_{\rbar}/\varpi$; recall that we say that it is \emph{effective} if all of
the multiplicities are non-negative.
\begin{thm}
  \label{thm: geometric BM}Let $\sigma$ be a non-Steinberg Serre weight. Then
  the cycle $\cC_{\sigma}$ is effective, 
and is nonzero
  precisely when $\sigma\in W(\rbar)$. 
  It is
  independent of the choice of integers $n_\tau(\sigma)$ satisfying~{\em
    (\ref{eqn: defn of n_tau(sigmabar)})}. For
  each tame type $\tau$, we
  have \[Z(R^\square_{\rbar,0,\tau}/\varpi)=\sum_{\sigma\in\JH(\sigmabar(\tau))}\cC_{\sigma}.\] 
\end{thm}
\begin{proof}
  We will argue exactly as in the proof of~\cite[Thm.\
  5.5.2]{emertongeerefinedBM} (taking $n=2$), and we freely use the
  notation and definitions of~\cite{emertongeerefinedBM}. Since
  $p>2$, 
  we have $p\nmid n$ and thus a suitable globalisation $\rhobar$ exists
  provided that~\cite[Conj.\ A.3]{emertongeerefinedBM} holds for
  $\rbar$. The latter follows from the proof of
  Theorem A.2 of~\cite{geekisin} (which shows that $\rbar$ has a
  potentially Barsotti--Tate lift) and Lemma 4.4.1 of \emph{op.cit}.\ 
  (which shows that any potentially Barsotti--Tate representation is
  potentially diagonalizable). These same results also show that the equivalent
  conditions of~\cite[Lem.\ 5.5.1]{emertongeerefinedBM} hold in the case that
  $\lambda_v=0$ for all $v$, and in particular in the case that $\lambda_v=0$
  and $\tau_v$ is tame for all $v$, which is all that we will require.

By~\cite[Lem.\ 5.5.1(5)]{emertongeerefinedBM}, we see that for each choice of
tame types $\tau_v$, we have 
\numequation\label{eqn: first BM eqn}
Z(\overline{R}_\infty/\varpi)=\sum_{\otimes_{v|p}\sigma_v}\prod_{v|p}m_{\sigma_v}(\tau_v)Z'_{\otimes_{v|p}\sigma_v}(\rhobar).\end{equation}
Now, by definition we
have  \numequation\label{eqn: second BM eqn}Z(\Rbarinfty/\varpi)=\prod_{v|p} Z(R^{\square}_{\rbar,0,\tau_v}/\varpi)\times Z
    (\F[[[x_1,\dots,x_{q-[F^+:\Q]n(n-1)/2},t_1,\dots,t_{n^2} ]]).\end{equation} Fix a
    non-Steinberg Serre weight $\sigma=\otimes_v\sigma_v$, and sum over
    all choices of types $\tau_v$, weighted by
    $\prod_{v|p}n_{\tau_v}(\sigma_v)$. We obtain \[\sum_\tau\prod_{v|p}n_{\tau_v}(\sigma_v)\prod_{v|p} Z(R^{\square}_{\rbar,0,\tau_v}/\varpi)\times Z
    (\F[[[x_1,\dots,x_{q-[F^+:\Q]n(n-1)/2},t_1,\dots,t_{n^2}
    ]])\]\[=\sum_\tau\prod_{v|p}n_{\tau_v}(\sigma_v)\sum_{\otimes_{v|p}\sigma'_v}\prod_{v|p}m_{\sigma'_v}(\tau_v)Z'_{\otimes_{v|p}\sigma'_v}(\rhobar) \]which
    by~(\ref{eqn: orthogonality of sigmabar tau}) simplifies to \numequation\label{eqn: third BM eqn}\prod_{v|p}\cC_{\sigma_v}\times Z
    (\F[[[x_1,\dots,x_{q-[F^+:\Q]n(n-1)/2},t_1,\dots,t_{n^2}
    ]])=Z'_{\otimes_{v|p}\sigma_v}(\rhobar).\end{equation} Since
    $Z'_{\otimes_{v|p}\sigma_v}(\rhobar)$ is effective by definition (as it
    is defined as a positive multiple of the support cycle of a patched module),
    this shows that every $\prod_{v|p}\cC_{\sigma_v}$ is effective. We conclude
    that either every $\cC_{\sigma}$ is effective, or that every
    $-\cC_{\sigma}$ is effective.

Substituting~(\ref{eqn: third BM eqn}) and~(\ref{eqn: second BM eqn}) into~(\ref{eqn:
  first BM eqn}), we see that \numequation\label{eqn: 4th BM eqn}\prod_{v|p} Z(R^{\square}_{\rbar,0,\tau_v}/\varpi)\times Z
    (\F[[[x_1,\dots,x_{q-[F^+:\Q]n(n-1)/2},t_1,\dots,t_{n^2} ]])\end{equation} \[=\prod_{v|p}\left(\sum_{\sigma\in\JH(\sigmabar(\tau))}\cC_{\sigma_v}\right)\times Z
    (\F[[[x_1,\dots,x_{q-[F^+:\Q]n(n-1)/2},t_1,\dots,t_{n^2}
    ]],\]and we deduce that either
    $Z(R^\square_{\rbar,0,\tau}/\varpi)=\sum_{\sigma}m_{\sigma}(\tau)\cC_{\sigma}$
    for all $\tau$, or
    $Z(R^\square_{\rbar,0,\tau}/\varpi)=-\sum_{\sigma}m_{\sigma}(\tau)\cC_{\sigma}$
    for all $\tau$.

Since each $Z(R^\square_{\rbar,0,\tau}/\varpi)$ is effective, the second
possibility holds if and only if every $-\cC_{\sigma}$ is effective (since
either all the $-\cC_{\sigma}$ are effective, or all the $\cC_{\sigma}$
are effective). It remains to show that this possibility leads to a
contradiction. Now, if
$Z(R^\square_{\rbar,0,\tau}/\varpi)=-\sum_{\sigma}m_{\sigma}(\tau)\cC_{\sigma}$
for all $\tau$, then substituting into the definition $\cC_{\sigma}=\sum_\tau
n_\tau(\sigma)Z(R^\square_{\rbar,0,\tau}/\varpi)$, we
obtain \[\cC_{\sigma}=\sum_{\sigma'}\left(\sum_{\tau}n_\tau(\sigma)m_{\sigma'}(\tau)\right)\left(-\cC_{\sigma'}\right),\]and
applying~(\ref{eqn: orthogonality of sigmabar tau}), we obtain
$\cC_{\sigma}=-\cC_{\sigma}$, so that $\cC_{\sigma}=0$ for all
$\sigma$.  Thus all the $\cC_{\sigma}$ are effective, as claimed.

Since $Z'_{\otimes_{v|p}\sigma_v}(\rhobar)$ by definition
depends only on (the global choices in the Taylor--Wiles method, and)
$\otimes_{v|p}\sigma_v$, and \emph{not} on the particular choice of the
$n_\tau(\sigma)$, it follows from~(\ref{eqn: third BM eqn})
that~$\cC_{\sigma}$ is also independent of this choice.

Finally, note that by definition
  $Z'_{\otimes_{v|p}\sigma_v}(\rhobar)$  is nonzero precisely when
  $\sigma_v$ is in the set $\WBT(\rbar)$ defined in~\cite[\S3]{geekisin}; but
  by the main result of~\cite{gls13}, this is precisely the set~$W(\rbar)$.
\end{proof}

\begin{rem}
  \label{rem: could have done BM for wildly ramified
    types}As we do not use wildly ramified types elsewhere in the
  paper, we have restricted the statement of Theorem~\ref{thm:
    geometric BM} to the case of tame types; but the statement
admits a natural extension
to the case
  of wildly ramified inertial types (with some components now occurring with multiplicity greater than
  one), and the proof goes through unchanged in this more general setting.
\end{rem}

\section{The geometric Breuil--M\'ezard conjecture for the stacks \texorpdfstring{$\cZ^{\dd,1}$}{Z\textasciicircum dd,1}}\label{sec: picture} We now prove our main results on the irreducible components of
$\cZ^{\dd,1}$. We do this by a slightly indirect method, defining
certain formal sums of these irreducible components which we then
compute via the geometric Breuil--M\'ezard conjecture, and in
particular the results of Section~\ref{sec:appendix on geom
  BM}. 

By Theorem~\ref{lem: C 1 and Z 1 are the underlying reduced substacks}, $\cZ^{\dd,1}$
is reduced and equidimensional, and each $\cZ^{\tau,1}$ is a union of some of
its irreducible components. Let $K(\cZ^{\dd,1})$ be the free abelian group
generated by the irreducible components of $\cZ^{\dd,1}$.   We say that an element of $K(\cZ^{\dd,1})$ is
\emph{effective} if the multiplicity of each irreducible
component is nonnegative.  We say that an element of $K(\cZ^{\dd,1})$ is
\emph{reduced and effective} if the multiplicity of each irreducible
component is $0$ or~$1$;\ we will sometimes abuse language by identifying a reduced and effective cycle with the reduced union of the irreducible components appearing in it.

Let $x$ be a finite type point of $\cZ^{\dd,1}$, corresponding to a
representation $\rbar:G_K\to\GL_2(\F')$, and recall that $R^{\dd,1}_x$ is a versal ring to $\cZ^{\dd,1}$ having each $R^{\tau,1}_x$ as a quotient. 
Since $\cZ^{\tau,1}$ is a union of irreducible
components of  $\cZ^{\dd,1}$, $\Spec R^{\tau,1}_x$ is a union of
irreducible components of~$\Spec R^{\dd,1}_x$. 

Let $K(R^{\dd,1}_x)$ be the free abelian group generated by the irreducible
components of $\Spec R^{\dd,1}_x$. By~\cite[\href{https://stacks.math.columbia.edu/tag/0DRB}{Tag 0DRB},\href{https://stacks.math.columbia.edu/tag/0DRD}{Tag 0DRD}]{stacks-project}, there is a
natural multiplicity-preserving surjection from the set of irreducible
components of $\Spec R^{\dd,1}_x$ to the set of irreducible components
of~$\cZ^{\dd,1}$ which contain~$x$. Using this surjection, we can
define a group homomorphism \[K(\cZ^{\dd,1})\to K(R^{\dd,1}_x)\]
in the following way: we send any irreducible component $\cZ$ of
$\cZ^{\dd,1}$ which contains~$x$ to the formal sum of the
irreducible components of~$\Spec R^{\dd,1}_x$ in the preimage of~$\cZ$ under
this surjection, and we send every other irreducible component to~$0$.

\begin{lem}
  \label{lem: can read off reduced and effective from local rings}An
  element~$\cTbar$ of~$K(\cZ^{\dd,1})$ is effective if
  and only if for every finite type point $x$ of~$\cZ^{\dd,1}$, the
  image of $\cTbar$ in $K(R^{\dd,1}_x)$ is effective. We have $\cTbar=0$ if
  and only if its image is $0$ in every $K(R^{\dd,1}_x)$.
\end{lem}
\begin{proof}
  The ``only if'' direction is trivial, so we need only consider the
  ``if'' implication. Write $\cTbar=\sum_{\cZbar}a_{\cZbar}\cZbar$,
  where the sum runs over the irreducible components $\cZbar$ of
  $\cZ^{\dd,1}$, and the $a_{\cZbar}$ are integers.

  Suppose first
  that the image of $\cTbar$ in $K(R^{\dd,1}_x)$ is effective for every $x$; we then have to
  show that each $a_{\cZbar}$ is nonnegative.
  To see this, fix an irreducible component~$\cZbar$, and choose~$x$
  to be a finite type point of $\cZ^{\dd,1}$ which is contained
  in~$\cZbar$ and in no other irreducible component
  of~$\cZ^{\dd,1}$. Then the image of~$\cTbar$ in $K(R^{\dd,1}_x)$ is equal to
  $a_{\cZbar}$ times a nonempty sum of irreducible components
  of~$\Spec R^{\dd,1}_x$. By hypothesis, this must be effective,
  which implies that $a_{\cZbar}$ is nonnegative, as required.

Finally, if the image of $\cTbar$ in $K(R^{\dd,1}_x)$ is $0$, then
$a_{\cZbar}=0$; so if this holds for all~$x$, then~$\cTbar=0$.
\end{proof}

For each tame type $\tau$, we let $\cZ(\tau)$ denote the formal sum of
the irreducible components of $\cZ^{\tau,1}$, considered as an
element of $K(\cZ^{\dd,1})$. By Lemma~\ref{lem:write Serre weight as
  linear combination of types}, for each non-Steinberg Serre weight
$\sigma$ of $\GL_2(k)$ there are integers $n_\tau(\sigma)$ such
that $\sigma=\sum_\tau n_\tau(\sigma)\sigmabar(\tau)$ in the
Grothendieck group of mod $p$ representations of $\GL_2(k)$, where the
$\tau$ run over the tame types. We set
\[\cZ(\sigma):=\sum_\tau n_\tau(\sigma)\cZ(\tau)\in K(\cZ^{\dd,1}).\]
The integers $n_\tau(\sigma)$ are not necessarily unique, but it
follows from the following result that $\cZ(\sigma)$  is independent
of the choice of $n_\tau(\sigma)$, and is reduced and effective.

\begin{thm}
  \label{thm:stack version of geometric Breuil--Mezard}\leavevmode
  \begin{enumerate}
  \item   Each $\cZ(\sigma)$  is an
    irreducible component of $\cZ^{\dd,1}$.
  \item The finite type points of $\cZ(\sigma)$ are precisely the
    representations $\rbar:G_K\to\GL_2(\F')$ having $\sigma$ as a
    Serre weight.
 \item For each tame type $\tau$, we have
   $\cZ(\tau)=\sum_{\sigma\in\JH(\sigmabar(\tau))}\cZ(\sigma)$. 
  \item Every irreducible component of $\cZ^{\dd,1}$ is of the form
    $\cZ(\sigma)$ for some unique non-Steinberg Serre weight~$\sigma$. 
    \item For each tame type $\tau$, and each $J\in \cP_\tau$,
	    we have $\cZ(\sigmabar(\tau)_J)=\overline{\cZ}(J)$.
  \end{enumerate}

\end{thm}
\begin{proof} Let~$x$ be a finite type point of $\cZ^{\dd,1}$
  corresponding to $\rbar:G_K\to\GL_2(\F')$, and
  write $\cZ(\sigma)_x$, $\cZ(\tau)_x$ for the images in $K(R^{\dd,1}_x)$ of
  $\cZ(\sigma)$ and $\cZ(\tau)$ respectively. Each $\Spec R^{\tau,1}_x$
  is a closed subscheme of $\Spec R^{\square}$, the universal framed deformation
  $\cO_{E'}$-algebra for~$\rbar$, so we may regard the $\cZ(\tau)_x$
  as formal sums (with multiplicities) of irreducible subschemes
  of~$\Spec R^{\square}/\pi$. 

  By definition, $\cZ(\tau)_x$ is just the underlying cycle of
  $\Spec R^{\tau,1}_x$. By Theorem~\ref{prop: generically reduced
    special fibre deformation ring}, this is equal to the underlying
  cycle of $\Spec R_{\rbar}^{\tau,\BT}/\varpi$. Consequently,
  $\cZ(\sigma)_x$ is the cycle denoted by~$\cC_{\sigma}$ in
  Section~\ref{sec:appendix on geom BM}. It follows from
  Theorem~\ref{thm: geometric BM} that:
  \begin{itemize}
  \item $\cZ(\sigma)_x$ is effective, and is nonzero precisely when
    $\sigma$ is a Serre weight for~$\rbar$.
  \item For each tame type $\tau$, we have
   $\cZ(\tau)_x=\sum_{\sigma\in\JH(\sigmabar(\tau))}\cZ(\sigma)_x$.
  \end{itemize}
Applying Lemma~\ref{lem: can read off reduced and effective from local
  rings}, we see that each $\cZ(\sigma)$ is effective, and that~(3)
holds. Since $\cZ^{\tau,1}$ is reduced, $\cZ(\tau)$ is reduced and
effective, so it follows from~(3) that each~$\cZ(\sigma)$ is
reduced and effective. Since $x$ is a finite type point of
$\cZ(\sigma)$ if and only if $\cZ(\sigma)_x\ne 0$, we have also proved~(2).

Since every irreducible component of~$\cZ^{\dd,1}$ is an irreducible
component of some~$\cZ^{\tau,1}$, in order to prove~(1) and~(4) it
suffices to show that for each~$\tau$, every irreducible component
of~$\cZ^{\tau,1}$ is of the form~$\cZ(\sigmabar(\tau)_J)$ for some~$J$, and
that each~$\cZ(\sigmabar(\tau)_J)$ is irreducible. Since $\tau$ is fixed for the rest of the argument, let us simplify notation by writing $\sigmabar_J$ for $\sigmabar(\tau)_J$.
Now, by
Theorem~\ref{thm:main thm cegsC}(2), we know
that~$\cZ^{\tau,1}$ has exactly~$\#\cP_\tau$ irreducible components,
namely the~$\cZbar(J')$ for~$J'\in\cP_\tau$. On the other hand,
the~$\cZ(\sigmabar_J)$ are reduced and effective, and since by the results
of \cite{gls13} there exist representations admitting~$\sigmabar_J$ as their
unique non-Steinberg Serre weight\footnote{More precisely, it follows from the results of \cite{gls13} that there exist $[K:\Q]$-dimensional extension spaces $V$ of reducible representations admitting~$\sigmabar_J$ as a Serre weight, such that the representations in $V$ admitting at least one other non-Steinberg weight lie in a finite union of proper subspaces;\ moreover the number of subspaces in this finite union is independent of $\F'$.}, it follows from~(2) that for each~$J$, there must
be a~$J'\in\cP_\tau$ such that~$\cZbar(J')$ contributes
to~$\cZ(\sigmabar_J)$, but not to any~$\cZ(\sigmabar_{J''})$
for~$J''\ne J$. 

Since~$\cZ(\tau)$ is reduced and effective, and the sum in~(3) is
over~$\#\cP_\tau$ weights~$\sigma$, it follows that we in fact have
$\cZ(\sigmabar_J)=\cZbar(J')$. This proves~(1) and~(4), and to
prove~(5), it only remains to show that~$J'=J$. To see this, note that
by~(2), $\cZ(\sigmabar_J)=\cZbar(J')$ has a dense open substack whose
finite type points have~$\sigmabar_J$ as their unique non-Steinberg Serre weight
(namely the complement of the union of the~$\cZ(\sigma')$ for
all~$\sigma'\ne\sigmabar_J$). By Theorem~\ref{thm:main thm cegsC}(4), 
it also has a dense open substack whose finite type points
have~$\sigmabar_{J'}$ as a Serre weight. Considering any finite type
point in the intersection of these dense open substacks, we see
that~$\sigmabar_J=\sigmabar_{J'}$, so that~$J=J'$, as required.
\end{proof}

\section{The geometric Breuil--M\'ezard conjecture for the stacks \texorpdfstring{$\cX_{2,\red}$}{X\_2,red}}\label{subsec:
  geometric BM book version}

We now explain how to transfer our results from the stacks $\cZ^{\dd,1}$ to the stacks $\cX_{2,\red}$ of \cite{EGmoduli}. The book \cite{EGmoduli} establishes an equivalence between the classical ``numerical'' Breuil--M\'ezard conjecture and the geometric Breuil--M\'ezard conjecture for the stacks $\cX_{2,\red}$.  (Indeed the implication from the former to the latter only requires the classical Breuil--M\'ezard conjecture at a single sufficiently generic point of each component of $\cX_{2,\red}$.) 

The Breuil--M\'ezard conjecture for two-dimensional potentially Barsotti--Tate respresentations is established in \cite{geekisin}, and this is extended to two-dimensional potentially semistable representations of Hodge type $0$ in \cite[Thm.~8.6.1]{EGmoduli}.  The arguments of \cite[\S8.3]{EGmoduli} translate this into the following theorem. Here $\cX_2^{\tau,\BT,\st}$ 
 denotes the stack of two-dimensional potentially semistable representations of Hodge type $0$ constructed in \cite{EGmoduli}, while $\sigma^{\st}(\tau)$ is as in \cite[Thm.~8.2.1]{EGmoduli}.

\begin{thm}[\cite{EGmoduli}]\label{thm:EG-BM} There exist effective cycles $Z^{\sigma}$ {\upshape{(}}elements of the free group on the irreducible components of $\cX_{2,\red}$, with nonnegative coefficients{\upshape{)}} such that for all inertial types $\tau$, 
\begin{itemize}
  \item the cycle of the special fibre of $\cX_{2}^{\tau,\BT}$ is equal to $\sum_{\sigma} m_{\sigma}(\tau) \cdot Z^{\sigma}$, while
  \item the cycle of the special fibre of $\cX_2^{\tau,\BT,\st}$ is equal to $\sum_{\sigma} m^{\st}_{\sigma}(\tau) \cdot Z^{\sigma}$.
\end{itemize}  Here $\sigmabar(\tau) = \sum_{\sigma} m_{\sigma}(\tau) \cdot \sigma$  and $\sigmabar^{\st}(\tau) = \sum_{\sigma} m^{\st}_{\sigma}(\tau) \cdot \sigma$ in the Grothendieck group of $\GL_2(k)$.
\end{thm}

\begin{cor}[\cite{EGmoduli}]\label{thm:EG-serre-weights}
Let $\rbar : G_K \to \GL_2(\F')$ be a continuous Galois representation, corresponding to a finite type point of $\cX_{2,\red}$. For each Serre weight $\sigma$ we have $\sigma \in W(\rbar)$ if and only if $\rbar$ lies in the support of $Z^{\sigma}$.
\end{cor}

\begin{proof}
This follows by the argument in \cite[\S8.4]{EGmoduli}:\ the Breuil--M\'ezard multiplicity $\mu_\sigma(\rbar)$ is nonzero if and only if $Z^{\sigma}$ is supported at $\rbar$. More precisely, invoking \cite[Thm.~8.6.1]{EGmoduli} in place of Conj.~8.2.2 of \emph{loc.\ cit.}, the next-to-last paragraph of \cite[\S8.4]{EGmoduli} shows that $\rbar$ lies in the support of $Z^{\sigma}$ if and only if $\sigma$ is an element of the weight set $W_{\BT}(\rbar) = \{ \sigma : \mu_{\sigma}(\rbar) > 0\}$. But $W_{\BT}(\rbar) = W(\rbar)$ by the main results of \cite{gls13}.
\end{proof}

\begin{remark}\label{rem:how-to-get-Z} The cycles $Z^{\sigma}$ of Theorem~\ref{thm:EG-BM} are constructed in \cite[\S8.3]{EGmoduli} as follows. For each Serre weight $\sigma'$, the  smooth points of $\cX^{\sigma'}_{2,\red}$ that furthermore do not lie on any other component of $\cX_{2,\red}$ are dense. Choose such a point $\rbar_{\sigma'}$ and let $\{\mu_{\sigma}(\rbar_{\sigma'})\}$ be the multiplicities in  the Breuil--M\'ezard conjecture for  $\rbar_{\sigma'}$. Then $Z^{\sigma} := \sum_{\sigma'} \mu_{\sigma}(\rbar_{\sigma'}) \cdot \cX^{\sigma'}_{2,\red}$.
\end{remark}

It remains to compute the cycles $Z^{\sigma}$. We begin with the following observation, which could  be proved  with modest effort by calculating dimensions of families of extensions and using the results of \cite{gls13}, but is also easily deduced from the results of Section~\ref{sec: picture}.

\begin{lemma}\label{lem:one-rep-with-one-serre-weight}
Let $\sigma,\sigma'$ be Serre weights and suppose that $\sigma'$ is non-Steinberg. Then $\cX^{\sigma}_{2,\red}$ contains at least one finite type point corresponding to a representation $\rbar$ with $\sigma' \not\in W(\rbar)$.
\end{lemma}

\begin{proof}
Suppose first that $\sigma$ is non-Steinberg. By Theorem~\ref{thm:stack version of geometric Breuil--Mezard} the component $\cZ(\sigma)$ of $\cZ^{\dd,1}$ has a dense open set $U(\sigma)$ whose finite type points $\rbar$ have no non-Steinberg Serre weights other than $\sigma$:\ take $\cZ(\sigma) \setminus \cup_{\sigma' \neq \sigma} \cZ(\sigma')$. The finite type points of $\cZ(\sigma)$ described in Remark~\ref{rem:maximalfamilies} are also dense;\  therefore at least one of them (indeed a dense set of them) lies in $U(\sigma)$.  Let $\rbar$ be such a representation. The finite type points of $\cZ(\sigma)$ described in Remark~\ref{rem:maximalfamilies} are precisely the family of niveau $1$ representations in the description of $\cX^{\sigma}_{2,\red}$ of Theorem~\ref{thm:intro-thm-EG} (and whose construction can be found in the proof of \cite[Thm~5.5.12]{EGmoduli} and its correction in the errata to \emph{loc.\ cit.}). So $\rbar$ lies on $\cX^{\sigma}_{2,\red}$ as well, and by construction the only non-Steinberg weight in $W(\rbar)$ is $\sigma$.  This completes the non-Steinberg case.

If instead $\sigma$ is Steinberg, then by construction, and by the second bullet point in \cite[Def.~5.5.1]{EGmoduli}, $\cX^{\sigma}_{2,\red}$ contains representations $\rbar$ that are tr\`es ramifi\'e. But tr\`es ramifi\'e representations have no non-Steinberg Serre weights by \cegsBtresramlemma.
\end{proof}

\begin{remark}\label{rem:promotion}
Once we have proved Theorem~\ref{thm:component-thm-in-last-sec} below, ``at least one finite type point'' in the statement of Lemma~\ref{lem:one-rep-with-one-serre-weight} can be promoted to ``a dense set of finite type points'' by taking $\cX^{\sigma}_{2,\red} \setminus \cup_{\sigma'\neq \sigma} \cX^{\sigma'}_{2,\red}$.
\end{remark}

We now reach our main theorem.

\begin{thm}\label{thm:component-thm-in-last-sec} Suppose $p > 2$. We have:
\begin{itemize}
  \item $Z^{\sigma} = \mathcal{X}_{2,\red}^\sigma$, if the weight $\sigma$ is not Steinberg, while
  \item $Z^{\chi\otimes \St} = \mathcal{X}_{2,\red}^\chi + \mathcal{X}_{2,\red}^{\chi \otimes \St}$ if the weight  $\sigma \cong \chi \otimes \St$ is Steinberg.
\end{itemize}
In particular $\sigma \in W(\rbar)$ if and only if $\rbar$ lies on $\cX^{\sigma}_{2,\red}$ if $\sigma$ is not Steinberg, or on $\mathcal{X}_{2,\red}^\chi \cup \mathcal{X}_{2,\red}^{\chi \otimes \St}$ if $\sigma \cong \chi \otimes \St$ is Steinberg.
\end{thm}

Essentially the same statement appears at \cite[Thm~8.6.2]{EGmoduli}, and the argument given there invokes an earlier version of this paper\footnote{As mentioned in the introduction, the reference [CEGS19, Thm.~5.2.2] in \cite{EGmoduli} is Theorem~\ref{thm:stack version of geometric Breuil--Mezard} of this paper, while the reference [CEGS19, Lem.~B.5] in \cite{EGmoduli} is \cegsBtresramlemma.}. The proof we give below is independent of the proof of \cite[Thm~8.6.2]{EGmoduli}, but has the same major beats, and is rearranged to delay until as late as possible making any references to the earlier parts of this paper. Indeed the proof  now only invokes the generic reducedness of Theorem~\ref{prop: generically reduced special fibre deformation ring} (but certainly invokes it in a crucial way);\ we hope that this clarifies the various dependencies involved. 

However, we also take the opportunity to repair a small gap in the argument given at \cite[Thm~8.6.2]{EGmoduli}. It is claimed there that [CEGS19, Thm.~5.2.2(2)] (i.e., our Theorem~\ref{thm:stack version of geometric Breuil--Mezard}(2)) implies that $\cX^{\sigma}_{2,\red}$ has a dense set of  finite type points corresponding to representations $\rbar$ whose only non-Steinberg Serre weight is $\sigma$. Although the conclusion is certainly true, the deduction is incorrect, or at least seems to presume that $\cX^{\sigma}_{2,\red}$ can be identified with our $\cZ(\sigma)$. We replace this claim with an argument using Lemma~\ref{lem:one-rep-with-one-serre-weight}, which does follow from Theorem~\ref{thm:stack version of geometric Breuil--Mezard} (or from the results of \cite{gls13}, as previously noted). The claim about  points of $\cX^{\sigma}_{2,\red}$ will follow once Theorem~\ref{thm:component-thm-in-last-sec} is proved, as explained in Remark~\ref{rem:promotion}.

\begin{proof}[Proof of Theorem~\ref{thm:component-thm-in-last-sec}]
Consider first the non-Steinberg case. The finite type points in the support of the effective cycle $Z^{\sigma}$ are precisely the representations having $\sigma$ as a Serre weight. By Lemma~\ref{lem:one-rep-with-one-serre-weight} each component $\cX^{\sigma'}_{2,\red}$ with $\sigma' \neq \sigma$ has a finite type point for which $\sigma$ is not a Serre weight;\ therefore $\cX^{\sigma'}_{2,\red}$ cannot occur in the cycle $Z^{\sigma}$. It follows that $Z^{\sigma} = \mu_{\sigma}(\rbar_{\sigma}) \cdot \cX^{\sigma}_{2,\red}$ with $\rbar_{\sigma}$ as in Remark~\ref{rem:how-to-get-Z}, and that $\mu_{\sigma}(\rbar_{\sigma'}) = 0$ for all $\sigma' \neq \sigma$. Note that we can already deduce that  $\sigma \in W(\rbar)$ if and only if $\rbar$ lies on $\cX^{\sigma}_{2,\red}$.

Now consider the Steinberg case;\ by twisting it will suffice to consider the weight $\sigma = \St$. The type $\sigma^{\st}(\triv)$ is the Steinberg type, and so by Theorem~\ref{thm:EG-BM} the cycle of the special fibre of $\cX^{\triv,\BT,\st}_{2}$ is equal to $Z^{\St}$. In particular the finite type points of $Z^{\St}$ are precisely the representations $\rbar$ having a  semistable lift of Hodge type~$0$. Such a lift is either crystalline, in which case the trivial Serre weight is a Serre weight for $\rbar$, and by the previous paragraph $\rbar$ is a finite type point of $\cX^{\triv}_{2,\red}$ ;\ or else the lift is semistable non-crystalline, in which case $\rbar$ is an unramified twist of an extension of the inverse of the cyclotomic character by the trivial character. Such an extension is either peu ramifi\'ee, in which case again $\triv \in W(\rbar)$, and $\rbar$ lies on $\cX^{\triv}_{2,\red}$ by the first paragraph of the proof;\ or else it is
tre\`s ramifi\'ee, and is a finite type point of $\cX^{\St}_{2,\red}$ (as a member of the  family of niveau $1$ representations defining $\cX^{\St}_{2,\red}$). Since all the finite type points of the support of $Z^{\St}$ are contained in $\cX^{\triv}_{2,\red} \cup \cX^{\St}_{2,\red}$ it follows that $Z^{\St} = \mu_{\St}(\rbar_{\triv})\cX^{\triv}_{2,\red} +  \mu_{\St}(\rbar_{\St})\cX^{\St}_{2,\red}$, and that $\mu_{\St}(\rbar_{\sigma'}) = 0$ for all $\sigma' \neq \triv,\St$.

In the remainder of the argument we can and do assume that $\rbar_{\triv}$ is an extension of an unramified twist of the inverse cyclotomic character by a \emph{different} unramified character:\ these are dense in $\cX^{\triv}_{2,\red}$ by the dominance of the eigenvalue morphism of \cite[Thm.~5.5.12(2)]{EGmoduli}.

It remains to determine the multiplicities $\mu_{\sigma}(\rbar_{\sigma})$, $\mu_{\St}(\rbar_{\triv})$, and $\mu_{\St}(\rbar_{\St})$. Each of these is positive because $\rbar_{\sigma}$ does have $\sigma$ as a Serre weight, while $\rbar_{\triv}$ has $\St$ as a Serre weight because each extension as in the previous paragraph has a crystalline lift with labelled Hodge--Tate weights $(p,0)$ at one embedding lifting each embedding $k \into \Fpbar$, and labelled Hodge-Tate weights $(1,0)$ at all other embeddings. 

Suppose that $\sigma$ is non-Steinberg. Choose any tame type $\tau$ such that $\sigmabar(\tau)$ has $\sigma$ as a Jordan--H\"older factor. The ring $R^{\tau,\BT}_{\rbar_{\sigma}}$ is versal to $\cX^{\tau,\BT}_2$ at $\rbar_{\sigma}$, and since $\rbar_{\sigma}$ is a smooth point of $\cX_{2,\red}$ the underlying reduced of $\Spec R^{\tau,\BT}_{\rbar_{\sigma}}/\varpi$ is smooth. But $\Spec R^{\tau,\BT}_{\rbar_{\sigma}}/\varpi$ is generically reduced by Theorem~\ref{prop: generically reduced special fibre deformation ring}. We deduce that the Hilbert--Samuel multiplicity of $R^{\tau,\BT}_{\rbar_{\sigma}}/\varpi$ is $1$, and therefore $\mu_{\sigma}(\rbar_\sigma) \le 1$. Since $\mu_{\sigma}(\rbar_\sigma)$ is positive it must be equal to $1$. Alternately, it follows from Theorem~\ref{thm:stack version of geometric Breuil--Mezard} and the isomorphism $\cZ^{\tau} \cong \mathcal{X}_{2}^{\tau,\BT}$ of \cite[Thm~4.5]{APAW} that the cycle of the special fiber of $\mathcal{X}_{2}^{\tau,\BT}$ is reduced and effective;\ since by Theorem~\ref{thm:EG-BM} this cycle contains $\mathcal{Z}^{\sigma}$ as a summand, we see again that $\mu_{\sigma}(\rbar_\sigma) \le 1$.

Our chosen $\rbar_{\triv}$ does not have any semistable non-crystalline lifts, and the semistable Hodge type $0$ deformation ring of $\rbar_{\triv}$ is simply a crystalline deformation ring, indeed one of the flat deformation rings studied by Kisin in \cite{kis04}. The argument in the previous paragraph showed that the Hilbert--Samuel multiplicity of $R^{\triv,\BT}_{\rbar_{\cris}}/\varpi$ is $1$.  It follows that $\mu_{\St}(\rbar_{\triv}) \le 1$, and since it is positive it must be precisely $1$. On the other hand the semistable Hodge type $0$ deformation ring of the tr\`es ramifi\'ee representation $\rbar_{\St}$ is an ordinary deformation ring, hence formally smooth, and we obtain $\mu_{\St}(\rbar_{\St}) = 1$.\end{proof}

\begin{remark}\label{rem:hi} The finite type points of $X^{\St}_{2,\red}$ are precisely those $\rbar$ having a semistable non-crystalline lift, i.e., the unramified twists of an extension of the inverse of the cyclotomic character by the trivial character;\ for the details see \cite[Lem.~8.6.4]{EGmoduli}.
\end{remark}

\renewcommand{\theequation}{\Alph{section}.\arabic{subsection}} 
\appendix

\section{A lemma on formal algebraic stacks}

   We suppose given a commutative diagram
   of morphisms of formal algebraic stacks
   $$\xymatrix{\cX \ar[r] \ar[dr] & \cY \ar[d] \\ & \Spf \cO}$$
         We suppose that each of $\cX$ and $\cY$ is quasi-compact
   and quasi-separated,
   and that the horizontal arrow is scheme-theoretically
   dominant, in the sense of \cite[Def.~6.13]{Emertonformalstacks}.
   We furthermore suppose that the morphism $\cX\to\Spf\cO$ realises
   $\cX$ as a finite type $\varpi$-adic formal algebraic stack.

   Concretely, 
   if we write $\cX^a := \cX\times_{\cO} \cO/\varpi^a,$
   then each $\cX^a$ is an algebraic stack,
   locally of finite type over $\Spec \cO/\varpi^a$, 
   and there is an isomorphism $\varinjlim_a \cX^a \iso \cX.$
   Furthermore, the assumption that the horizontal arrow
   is scheme-theoretically dominant means that we may
   find an isomorphism $\cY \cong \varinjlim_a \cY^a$,
   with each $\cY^a$ being a quasi-compact and quasi-separated
   algebraic stack, and with the transition morphisms being
   thickenings, such that the morphism $\cX \to \cY$
   is induced by a compatible family of morphisms
   $\cX^a \to \cY^a$, each of which is scheme-theoretically
   dominant. (The~$\cY^a$ are uniquely determined by the
         requirement that for all~$b\ge a$ large enough so that the
         morphism $\cX^a\to\cY$ factors
         through~$\cY\otimes_\cO\cO/\varpi^b$, $\cY^a$ is the
         scheme-theoretic image of the morphism
         $\cX^a\to\cY\otimes_\cO\cO/\varpi^b$. In particular, $\cY^a$
         is a closed substack of $\cY\times_\cO \cO/\varpi^a$.) 

   It is often the case, in the preceding situation,
   that $\cY$ is also a $\varpi$-adic formal algebraic stack.
   For example, we have the following result.
   (Note that the usual graph argument
   shows that the morphism $\cX\to \cY$ is necessarily
     algebraic, i.e.\ representable by algebraic 
  stacks, in the sense 
of~\cite[\href{http://stacks.math.columbia.edu/tag/06CF}{Tag 06CF}]{stacks-project} and \cite[Def.~3.1]{Emertonformalstacks}. Thus it makes sense to 
speak of it being proper, following \cite[Def.~3.11]{Emertonformalstacks}.)

   \begin{aprop}
     \label{prop:Y is p-adic formal}
     Suppose that the morphism $\cX \to \cY$ is proper,
and that $\cY$ is locally Ind-finite type over $\Spec \cO$
{\em (}in the sense of {\em \cite[Rem.~8.30]{Emertonformalstacks})}.
Then $\cY$ is a $\varpi$-adic formal algebraic stack.
   \end{aprop}
   \begin{proof}
     This is an application of \cite[Prop.~10.5]{Emertonformalstacks}.
   \end{proof}
  
A key point is that, because the formation of scheme-theoretic images
is not generally compatible with non-flat base-change,
the closed immersion
\anumequation
\label{eqn:closed immersion}
\cY^a \hookrightarrow \cY \times_{\cO} \cO/\varpi^a
\end{equation}
is typically {\em not} an isomorphism, even if $\cY$ is a $\varpi$-adic
formal algebraic stack.
Our goal in the remainder of this discussion is to give a criterion
(involving the morphism $\cX\to\cY$)
on an open substack $\cU \hookrightarrow \cY$ which guarantees
that the closed immersion
$\cU\times_{\cY}\cY^a \hookrightarrow \cU\times_{\cO}\cO/\varpi^a$
induced by~(\ref{eqn:closed immersion})
{\em is} an isomorphism.

We begin by establishing a simple lemma.
         For any $a \geq 1$,
   we have the $2$-commutative diagram
   \anumequation
   \label{eqn:base-change}
   \xymatrix{
   \cX^a \ar[r] \ar[d] & \cY^a \ar[d] \\
   \cX \ar[r] & \cY 
 }
         \end{equation}
   Similarly, if $b \geq a \geq 1$,
   then we have the $2$-commutative diagram
   \anumequation
   \label{eqn:base-change bis}
   \xymatrix{
   \cX^a \ar[r] \ar[d] & \cY^a \ar[d] \\
   \cX^b \ar[r] & \cY^b 
 }
         \end{equation}

   \begin{alemma}
     \label{lem:cartesian diagrams}
     Each of the diagrams~{\em (\ref{eqn:base-change})}
     and~{\em (\ref{eqn:base-change bis})} 
     is $2$-Cartesian.
   \end{alemma}
   \begin{proof}
    We may embed the diagram~(\ref{eqn:base-change})
         in the larger $2$-commutative diagram
       $$\xymatrix{\cX^a \ar[r] \ar[d] & \cY^a \ar^-{\text{(\ref{eqn:closed
               immersion})}}[r] \ar[d] &\cY\otimes_{\cO}
         \cO/\varpi^a \ar[d] \\ \cX \ar[r] &\cY \ar@{=}[r]& \cY}$$
       Since the outer rectangle is manifestly $2$-Cartesian,
       and since~(\ref{eqn:closed immersion}) is a closed immersion (and thus
       a monomorphism), we conclude that~(\ref{eqn:base-change}) is indeed
       $2$-Cartesian.  

       A similar argument shows that~(\ref{eqn:base-change bis}) 
       is $2$-Cartesian.
   \end{proof}

         We next note that,
   since each of the closed immersions $\cY^a \hookrightarrow \cY$
   is a thickening, giving an open substack $\cU \hookrightarrow \cY$
   is equivalent to giving an open substack $\cU^a \hookrightarrow \cY^a$
   for some, or equivalently, every, choice of $ a\geq 1$;
   the two pieces of data are related by the formulas
   $\cU^a := \cU\times_{\cY} \cY^a$
   and $\varinjlim_a \cU^a \iso \cU.$

   \begin{lemma}
     \label{prop:opens}
     Suppose that $\cX \to \cY$ is proper.
     If $\cU$ is an open substack of~$\cY$,
     then the following conditions are equivalent:
     \begin{enumerate}
       \item
         The morphism $\cX\times_{\cY} \cU \to \cU$
         is a monomorphism.
       \item
         The morphism $\cX\times_{\cY} \cU \to \cU$
         is an isomorphism.
       \item   For every $a \geq 1$,
         the morphism $\cX^a\times_{\cY^a} \cU^a \to
         \cU^a$
         is a monomorphism.
       \item   For every $a \geq 1$,
         the morphism $\cX^a\times_{\cY^a} \cU^a \to
         \cU^a$
         is an isomorphism.
       \item   For some $a \geq 1$,
         the morphism $\cX^a\times_{\cY^a} \cU^a \to
         \cU^a$
         is a monomorphism.
       \item   For some $a \geq 1$,
         the morphism $\cX^a\times_{\cY^a} \cU^a \to
         \cU^a$
         is an isomorphism.
       \end{enumerate}
       Furthermore, if these equivalent conditions hold,
       then the closed immersion  $\cU^a \hookrightarrow
       \cU\times_{\cO} \cO/\varpi^a$ is an isomorphism,
       for each $a \geq 1$.
   \end{lemma}
   \begin{proof}
     The key point is that 
     Lemma~\ref{lem:cartesian diagrams} implies that the diagram
     $$\xymatrix{\cX^a\times_{\cY^a}\cU^a \ar[r] \ar[d] & \cU^a \ar[d] \\
       \cX\times_{\cY} \cU \ar[r] & \cU }
     $$
     is $2$-Cartesian, for any $a \geq 1$,
     and similarly, that if $b \geq a \geq 1,$ then the diagram
     $$\xymatrix{\cX^a \times_{\cY^a}\cU^a\ar[r] \ar[d] & \cU^a \ar[d] \\
       \cX^b\times_{\cY^b} \cU^b \ar[r] & \cU^b }
     $$
     is $2$-Cartesian.
     Since the vertical arrows of this latter diagram
     are finite order thickenings,
     we find (by applying the analogue
     of~\cite[\href{http://stacks.math.columbia.edu/tag/09ZZ}{Tag 09ZZ}]{stacks-project}
     for algebraic stacks, whose straightforward deduction
     from that result we leave to the reader)
     that the top horizontal arrow
     is a monomorphism
     if and only if the bottom horizontal arrow is.
     This shows the equivalence of~(3) and~(5).
     Since the morphism $\cX\times_{\cY} \cU \to \cU$
     is obtained as the inductive limit
     of the various morphisms $\cX^a\times_{\cY^a} \cU^a
     \to \cU^a,$
     we find that~(3) implies~(1) (by applying e.g.\
     \cite[Lem.~4.11~(1)]{Emertonformalstacks}, which
     shows that the inductive limit of monomorphisms
     is a monomorphism),
     and also that~(4) implies~(2) (the inductive
     limit of isomorphisms being again an isomorphism).

     Conversely, if~(1) holds, then the base-changed
                 morphism 
     $$\cX\times_{\cY} (\cU\times_{\cO} \cO/\varpi^a) \to \cU
     \times_{\cO} \cO/\varpi^a$$
     is a monomorphism. 
     The source of this morphism admits an alternative description
     as $\cX^a \times_{\cY} \cU,$ which the
     $2$-Cartesian diagram
     at the beginning of the proof allows us to identify
     with 
     $\cX^a\times_{\cY^a} \cU^a$. 
     Thus we obtain a monomorphism
     $$\cX^a\times_{\cY^a} \cU^a \hookrightarrow \cU\times_{\cO}
     \cO_/\varpi^a.$$

     Since this monomorphism factors through
     the closed immersion $\cU^a \hookrightarrow \cU
     \times_{\cO} \cO/\varpi^a,$
     we find that each of the morphisms of~(3) is a monomorphism;
     thus~(1) implies~(3).  Similarly, (2)~implies~(4),
     and also implies 
     that the closed immersion
     $\cU^a \hookrightarrow \cU\times_{\cO} \cO/\varpi^a$
     is an isomorphism, for each $a \geq 1$.

     Since clearly (4) implies~(6),
     while (6) implies~(5), to complete the proof
     of the proposition, it suffices to show that
     (5) implies~(6).
     Suppose then that
     $\cX^a\times_{\cY^a} \cU^a \to \cU^a$
     is a monomorphism.  Since $\cU^a \hookrightarrow \cY^a$
     is an open immersion, it is in particular flat.  Since
     $\cX^a \to \cY^a$ is scheme-theoretically dominant and
     quasi-compact (being proper), any flat base-change of this
     morphism is again scheme-theoretically dominant,
     as well as being proper.
     Thus we see that $\cX^a\times_{\cY^a} \cU^a \to \cU^a$ is a
     scheme-theoretically dominant proper monomorphism,
     i.e.\ a scheme-theoretically dominant closed immersion,
     i.e.\ an isomorphism, as required.
   \end{proof}

\bibliographystyle{amsalpha-custom}
\bibliography{dieudonnelattices}
\end{document}